\documentclass[12pt]{amsart}
\usepackage{ amsmath,amssymb,amsthm,color, euscript, enumerate,amsfonts}
\usepackage{comment}
\usepackage{dsfont}
\usepackage{tikz}
\usepackage{longtable}

\newcommand{\1}{ \mathds{1}}

\newcommand{\maru}[1]{{\ooalign{\hfil#1\/\hfil\crcr
\raise.167ex\hbox{\mathhexbox20D}}}}

\newcommand{\ruby}[2]{%
 \leavevmode
 \setbox0=\hbox{#1}%
 \setbox1=\hbox{\tiny #2}%
 \ifdim\wd0>\wd1 \dimen0=\wd0 \end{lemma}se \dimen0=\wd1 \fi
 \hbox{%
   \kanjiskip=0pt plus 2fil
   \xkanjiskip=0pt plus 2fil
   \vbox{%
     \hbox to \dimen0{%
       \tiny \hfil#2\hfil}%
     \nointerlineskip
     \hbox to \dimen0{\mathstrut\hfil#1\hfil}}}}

\newcommand{\Z}{\mathbb{Z}}
\newcommand{\C}{\mathbb{C}}
\newcommand{\R}{\mathbb{R}}

\newcommand{\Q}{\mathbb{Q}}

\newcommand{\g}{\mathfrak{g}}

\newcommand{\Fg}{\mathfrak{g}}

\newcommand{\h}{\mathfrak{h}}

\newcommand{\End}{\mathrm{End}}

\newcommand{\Aut}{\mathrm{Aut}\,}

\newcommand{\ad}{\mathrm{ad}}

\newcommand{\p}{\varphi}

\newcommand{\Hom}{\mathrm{Hom}}
\newcommand{\w}{\omega}

\makeatletter \@addtoreset{equation}{section}

\theoremstyle{plain}
\newtheorem{theorem}{Theorem}[section]
\newtheorem{proposition}[theorem]{Proposition}
\newtheorem{lemma}[theorem]{Lemma}

\theoremstyle{definition}
\newtheorem{definition}[theorem]{Definition}

\theoremstyle{remark}
\newtheorem{remark}[theorem]{Remark}
\numberwithin{equation}{section}

\title[Orbifold construction associated with the Leech lattice VOA]{On orbifold constructions associated with the Leech lattice vertex operator algebra}
\author{Ching Hung Lam} %
  \address[C. H. Lam] {Institute of Mathematics, Academia Sinica, Taipei 10617, Taiwan and National Center for Theoretical Sciences of  Taiwan.}
  \email{chlam@math.sinica.edu.tw}
\author[H. Shimakura]{Hiroki Shimakura}%
\address[H. Shimakura]{Graduate School of Information Sciences,
Tohoku University,
Sendai 980-8579, Japan }%
\email {shimakura@m.tohoku.ac.jp}%
\date{}
\thanks{C.\,H. Lam was partially supported by MoST grant 104-2115-M-001-004-MY3 of Taiwan}
\thanks{H.\ Shimakura was partially supported by JSPS KAKENHI Grant Numbers 26800001 and 17K05154.}
\thanks{C.\,H. Lam and H.\ Shimakura were partially supported by JSPS Program for Advancing Strategic International Networks to Accelerate the Circulation of Talented Researchers ``Development of Concentrated Mathematical Center Linking to Wisdom of the Next Generation".}

\newcommand{\sfr}[2]{\leavevmode\kern-.1em
  \raise.5ex\hbox{\the\scriptfont0 #1}\kern-.1em
  /\kern-.15em\lower.25ex\hbox{\the\scriptfont0 #2}}

\begin{document}

\begin{abstract}
In this article, we study orbifold constructions associated with the Leech lattice vertex operator algebra.
As an application, we prove that the structure of a strongly regular holomorphic vertex operator algebra of central charge $24$ is uniquely determined by its weight one Lie algebra if the Lie algebra has the type $A_{3,4}^3A_{1,2}$, $A_{4,5}^2$, $D_{4,12}A_{2,6}$, $A_{6,7}$, $A_{7,4}A_{1,1}^3$, $D_{5,8}A_{1,2}$ or $D_{6,5}A_{1,1}^2$ by using the reverse orbifold construction.
Our result also provides alternative constructions of these vertex operator algebras (except for the case $A_{6,7}$) from the Leech lattice vertex operator algebra.
\end{abstract}
\maketitle


\section{Introduction}
In this article, we continue our program on classification of holomorphic vertex operator algebras (VOAs) of central charge $24$ based on the  $71$ possible weight one Lie algebra structures in Schellekens' list (\cite{Sc93,EMS}).
This program can be divided into two parts--- the existence and the uniqueness parts.
Recently, the existence part, that is, constructions of $71$ holomorphic VOAs of central charge $24$, has been established (\cite{FLM,Bo,DGM,Lam,LS12,Mi3,SS,LS16,LS16b,EMS,LLin}).
The remaining question is to prove that the holomorphic VOA structure is uniquely determined by its weight one Lie algebra if the central charge is $24$.
Up to now, the uniqueness has been established for $57$ cases in \cite{DM,LS15,LS,LLin,KLL,EMS2}.

In \cite{LS}, a general method for proving the uniqueness of a holomorphic VOA $V$ with $V_1\neq 0$ has been proposed (see Section 2.6 for detail). 
Roughly speaking, the main idea is to ``reverse" the original orbifold construction and to reduce the uniqueness problem of holomorphic VOAs to the uniqueness  of some conjugacy classes of the automorphism groups of some ``known" holomorphic VOAs, such as lattice VOAs. In particular, explicit knowledge about the full automorphism groups of the ``known" VOAs plays an important role in this method.

In this article, we will establish the uniqueness for more cases. Our approach based on the Leech lattice VOA is motivated by the case $A_{6,7}$. In \cite{LS16b}, a holomorphic VOA $V$ of central charge $24$ with $V_1=A_{6,7}$ was constructed  by applying a  $\Z_7$-orbifold construction to the Leech lattice VOA and a non-standard lift of an order $7$ isometry of the Leech lattice.
The case $A_{6,7}$ is somewhat special because it is the only case in Schellekens' list that contains an affine VOA of level $7$. If we apply an orbifold construction to a holomorphic VOA $V$ of central charge $24$ with $V_1=A_{6,7}$ and a suitable automorphism of finite order, it seems that the weight one Lie algebra of the resulting VOA is either abelian or isomorphic to $A_{6,7}$ (cf. \cite[Proposition 5.5]{LS16}). Therefore, in some sense, the construction in \cite{LS16b} from the Leech lattice VOA is the only way for obtaining a holomorphic VOA with the weight one Lie algebra $A_{6,7}$ by using orbifold constructions.

In order to apply the method in \cite{LS} and to prove the uniqueness for the case $A_{6,7}$,
 we will try to ``reverse" the above orbifold construction associated with the Leech lattice VOA. Namely, we should define an automorphism $\sigma$ of a holomorphic VOA $V$ with $V_1=A_{6,7}$ so that the Leech lattice VOA is obtained by applying the orbifold construction to $V$ and $\sigma$. 
Indeed, we will define $\sigma$ as an order $7$ inner automorphism such that the restriction to the weight one Lie algebra $V_1$ is regular, that is, the fixed-point subalgebra is abelian.
Since the weight one Lie algebra of the Leech lattice VOA is abelian, it seems that such an automorphism of $V$ is the only possible choice. 
We can easily confirm the necessary conditions on $\sigma$ for the $\Z_7$-orbifold construction.
In addition, we will show that the orbifold construction associated with $V$ and $\sigma$ actually gives the Leech lattice VOA;
it is enough to prove that the dimension of $V_1$ is $24$ by Schellekens' list.
This will be verified by using the dimension formulae on $V_1$ (\cite{Mon,Mo,EMS2}).
In order to apply the formulae, we prove that for $1\le i\le 6$, the conformal weight of the irreducible $\sigma^i$-twisted $V$-module is at least $1$ by using a similar combinatorial argument as in \cite{LS}.

The remaining task is to prove the uniqueness of the conjugacy class of the automorphism $\varphi$ in the automorphism group $\Aut V_\Lambda$ of the Leech lattice VOA $V_\Lambda$ under the assumption that the orbifold construction associated with $V_\Lambda$ and $\varphi$ gives the original holomorphic VOA $V$.
By \cite{DN}, $\Aut V_\Lambda$ is an extension of the isometry group $O(\Lambda)$ of the Leech lattice $\Lambda$ by an abelian group $(\C^\times)^{24}$.
Hence we have $\varphi=\sigma\phi_g$ for some inner automorphism $\sigma\in (\C^\times)^{24}$ and a standard lift $\phi_g$ of an isometry $g$ of $\Lambda$.
The assumption on $\varphi$ gives some constraints, such as dimension or weights for $(V_\Lambda^\varphi)_1$-modules, on the weight one subspaces of the fixed-point subalgebra and the irreducible $\varphi$-twisted $V_\Lambda$-module.
These constraints turn out to be sufficient to determine the conjugacy class of $g$ in $O(\Lambda)$ uniquely.
In addition, we verify that, under the constraints above, $\sigma\in(\C^\times)^{24}$ is unique up to conjugation by the centralizer $C_{O(\Lambda)}(g)$ of $g$ in $O(\Lambda)$.
Thus $\varphi$ belongs to the unique conjugacy class in $\Aut V_\Lambda$.

By the structure of $\Aut V_\Lambda$, it is easier to handle this group than the automorphism group of the other Niemeier lattice VOA, which is an advantage for using the Leech lattice VOA.
Indeed, the uniqueness of conjugacy classes of this group can be verified by calculations on the Leech lattice and its isometry group, the Conway group.
The technical details on finite order automorphisms of (semi)simple Lie algebras in \cite{Kac} (cf.\ \cite{LS,EMS2}) are not necessary in our argument.
Instead, we use some known facts about the Conway group and the Leech lattice (cf.\ \cite{Wi83,HL90,HM16}).

In addition to the case $A_{6,7}$, we also establish the uniqueness for six other cases by the same manner: $A_{3,4}^3A_{1,2}$, $A_{4,5}^2$, $D_{4,12}A_{2,6}$, $A_{7,4}A_{1,1}^3$, $D_{5,8}A_{1,2}$ and $D_{6,5}A_{1,1}^2$.
Our main result is as follows. 
\begin{theorem}\label{Thm:main} The structure of a strongly regular holomorphic vertex operator algebra of central charge $24$ is uniquely determined by its weight one Lie algebra if the Lie algebra has the type $A_{3,4}^3A_{1,2}$, $A_{4,5}^2$, $D_{4,12}A_{2,6}$, $A_{6,7}$, $A_{7,4}A_{1,1}^3$, $D_{5,8}A_{1,2}$ or $D_{6,5}A_{1,1}^2$.
\end{theorem}
Our result also implies that holomorphic VOAs whose weight one Lie algebras have the type $A_{3,4}^3A_{1,2}$, $A_{4,5}^2$, $D_{4,12}A_{2,6}$, $A_{7,4}A_{1,1}^3$, $D_{5,8}A_{1,2}$ and $D_{6,5}A_{1,1}^2$ can be constructed from the Leech lattice VOA by orbifold constructions.
We remark that the uniqueness for the cases $A_{3,4}^3A_{1,2}$ and $A_{4,5}^2$ are proved in \cite{EMS2} by using the orbifold construction from the Niemeier lattices with root lattice $D_4^6$ and $A_4^6$, respectively.
Hence there are still $9$ Lie algebras, including $V_1=0$ case, that the corresponding uniqueness result has not been established yet (see Remark \ref{R:Nils} for explicit types). 

\medskip

The organization of the article is as follows:
In Section 2, we review some preliminary results about integral lattices, Lie algebras and VOAs.
In Section 3, we review some facts about conjugacy classes of the Conway group and sublattices of the Leech lattice; these facts are verified by the computer algebra system MAGMA (\cite{MAGMA}).
In Section 4, we review some basic properties of lattice VOAs, their automorphism groups and irreducible twisted modules.
We also discuss (standard) lifts of isometries of a lattice in the automorphism group of the lattice VOA.
In Section 5, we prove the uniqueness of certain conjugacy classes of the automorphism group of the Leech lattice VOA under some assumptions on the fixed-point subspaces and irreducible twisted modules. The proofs are based  on the results in Section 3.
In Section 6, we prove the main theorem by using the reverse orbifold construction (Section 2.6) and the results in Section 5.

\newpage

\begin{center}
{\bf Notations}
\begin{small}
\begin{longtable}{ll}\\
$(\cdot|\cdot)$& the positive-definite symmetric bilinear form of a lattice, or \\
& the normalized Killing form so that $(\alpha|\alpha)=2$ for any long root $\alpha$.\\
$\langle\cdot|\cdot\rangle$& the normalized symmetric invariant bilinear form on a VOA $V$\\
& so that $\langle \1|\1\rangle=-1$, equivalently, $\langle a|b\rangle\1=a_{(1)}b$ for $a,b\in V_1$.\\
$L^g$, $L_g$& for an isometry $g$ of $L$, $L^g=\{v\in L\mid g(v)=v\}$ and $L_g=\{v\in L\mid (v|L^g)=0\}$.\\
$L_\mathfrak{g}(k,0)$& the simple affine VOA associated with a simple Lie algebra $\mathfrak{g}$ at level $k$.\\
$L_\Fg(k,\lambda)$& the irreducible $L_\Fg(k,0)$-module with the highest weight $\lambda$.\\
$\Lambda$& the Leech lattice.\\
$\Lambda(g,p,q)$& $\{x+P_0^g(\Lambda)\in p\Lambda^g/P_0^g(\Lambda)\mid (x+P_0^g(\Lambda))(q)\neq\emptyset\}$ for $g\in O(\Lambda)$ and $p,q\in\Q$,\\
& where $(x+P_0^g(\Lambda))(q)=\{y\in x+P_0^g(\Lambda)\mid (y|y)=q\}$.\\
$M^{(u)}$& the $\sigma_u$-twisted $V$-module constructed from a $V$-module $M$ by Li's $\Delta$-operator.\\
$\mu$& the canonical surjective map from $\Aut V_L$ to $O(L)$ when $L$ has no roots.\\
$\mathcal{M}(S)$& the square matrix indexed by a set $S\subset\R^m$ with the entry $\frac{2(x|y)}{(x|x)}$, $x,y\in S$.\\
$O(L)$& the isometry group of a lattice $L$.\\
$\phi_g$& a standard lift of an isometry $g$ of $L$ to $O(\hat{L})$.\\
$P_0^g$& the orthogonal projection from $\R\otimes_\Z L$ to $\R\otimes_\Z L^g$.\\
$\Pi(M)$& the set of $\h$-weights of a module $M$\\
& for a reductive Lie algebra and a Cartan subalgebra $\h$.\\
$\sigma_u$& the inner automorphism $\exp(-2\pi\sqrt{-1}u_{(0)})$ of a VOA $V$ associated with $u\in V_1$.\\
$V^{\sigma}$& the set of fixed-points of an automorphism $\sigma$ of a VOA $V$.\\
$V[\sigma]$ & the irreducible $\sigma$-twisted module for a holomorphic VOA $V$.\\
$\tilde{V}_\sigma$& the VOA obtained by the orbifold construction associated with $V$ and $\sigma$.\\
$X_{n,k}$ & (the type of) a simple Lie algebra whose type is $X_n$ and level is $k$.\\
\end{longtable}
\end{small}
\end{center}
\section{Preliminary}
In this section, we will review basics about integral lattices, Lie algebras and VOAs.

\subsection{Even lattices}\label{S:lattice}

Let $(\cdot|\cdot)$ be a positive-definite symmetric bilinear form on $\R^m$.
A subset $L$ of $\R^m$ is called a \emph{lattice} of rank $m$ if $L$ has a basis $e_1,e_2,\dots,e_m$ of $\R^m$ satisfying $L=\bigoplus_{i=1}^m\Z e_i$.
Let $L^*$ denote the dual lattice of a lattice $L$ of rank $m$, that is, $$L^*=\{v\in \R^m\mid ( v| L)\subset\Z\}. $$
A lattice $L$ is said to be \emph{even} if $( v|v)\in2\Z$ for all $v\in L$, and is said to be \emph{unimodular} if $L=L^*$.
Note that any even lattice $L$ is integral, i.e., $(v|w)\in\Z$ for all $v,w\in L$.
For $v\in \R^m$, we call $(v|v)$ the (squared) \emph{norm} of $v$ and often denote it by $|v|^2$.

\begin{lemma}\label{Lem:lattice1} Let $L$ be a lattice and let $M$ be a direct summand sublattice of $L$ as an abelian group.
Let $N=\{v\in L\mid (v|M)=0\}$.
Then for any $v\in M^*$, there exists $x\in L^*$ such that $v-x\in N^*$.
In particular, $M^*$ is equal to the image of $L^*$ under the orthogonal projection from $\R\otimes_\Z L$ to $\R\otimes_\Z M$.
\end{lemma}
\begin{proof} Let $v\in M^*$.
Since $M$ is a direct summand of $L$, there exists $x\in L^*$ such that $(v|\alpha)=(x|\alpha)$ for all $\alpha\in M$.
Hence $(v-x|M)=0$.
It follows from $(v|N)=0$ and $(x|N)\subset\Z$ that $v-x\in N^*$. 
\end{proof}

Let $L$ be a lattice.
A group automorphism $g$ of $L$ is called an \emph{isometry} of $L$ if $(g(v)|g(w))=(v|w)$ for all $v,w\in L$.
Let $O(L)$ denote the isometry group of $L$.
For $g\in O(L)$, let $L^g$ denote the fixed-point set of $g$, that is, $L^g=\{v\in L\mid g(v)=v\}$.
Clearly $L^g$ is a sublattice of $L$, and it is a direct summand of $L$ as an abelian group.
Let $P_0^g$ denote the orthogonal projection from $\R\otimes_\Z L$ to $\R\otimes_\Z L^g$, i.e.,\begin{equation}
P_0^{g}=\frac{1}{n}\sum_{i=0}^{n-1}g^i,\label{Eq:OP}
\end{equation}
where $n$ is the order of $g$.
The following lemma is immediate from Lemma \ref{Lem:lattice1}.
\begin{lemma}\label{L:P0} 
Let $L$ be an even unimodular lattice and $g\in O(L)$.
Then $P_0^g(L)=(L^g)^*$.
\end{lemma}

\subsection{Regular automorphisms of simple Lie algebras}\label{regularauto}
Let $\mathfrak{g}$ be a simple finite-dimensional Lie algebra over the complex field $\C$.
Take a Cartan subalgebra $\h$.
Let $(\cdot|\cdot)$ be the normalized Killing form on $\g$ such that $(\alpha|\alpha)=2$ for any long root $\alpha$.
Here we identify $\h$ with $\h^*$ via $(\cdot|\cdot)$.
For a root $\beta$, the vector $\beta^\vee=\frac{2\beta}{(\beta|\beta)}$ is called a coroot.
Fix the simple roots $\alpha_1,\dots,\alpha_m$.
The fundamental weights $\Lambda_1,\dots,\Lambda_m$ (resp. the fundamental coweights $\Lambda_1^\vee,\dots,\Lambda_m^\vee$) are defined by $(\alpha_i^\vee|\Lambda_j)=\delta_{i,j}$ (resp. $(\alpha_i|\Lambda_j^\vee)=\delta_{i,j}$) for all $i,j$.

\begin{definition}
A finite order automorphism of a (semi)simple finite-dimensional Lie algebra (over $\C$) is said to be \emph{regular} if the fixed-point Lie subalgebra is abelian.
\end{definition}

\begin{lemma}[{\cite[Exercise 8.11]{Kac}}]\label{Lem:Kacfpa}
The minimum order of a regular automorphism of a simple finite-dimensional Lie algebra is the Coxeter number $h$.
Moreover, 
\begin{equation}
\exp(2\pi\sqrt{-1}\ad(\frac{1}{h}\tilde\rho))\label{Eq:reg}
\end{equation}
 is a unique order $h$ regular automorphism, up to conjugation, where $\tilde\rho=\sum_{i=1}^m\Lambda_i^\vee$.
\end{lemma}

Next, we will consider some matrix related to a finite set in $\R^m$:

\begin{definition}\label{D:M} Let $S$ be a finite set in $\R^m\setminus\{0\}$.
We define $\mathcal{M}(S)$ to be the square matrix indexed by $S$ with the entry $\frac{2(x|y)}{(x|x)}$, $x,y\in S$.
For finite sets $S,S'\subset\R^m\setminus\{0\}$, the matrices $\mathcal{M}(S)$ and $\mathcal{M}(S')$ are \emph{equivalent} if $|S|=|S'|$ and there exists a permutation matrix $P$ such that $\mathcal{M}(S)=P^{-1}\mathcal{M}(S')P$.
\end{definition}

For example, if $S$ is the set of simple roots of a root system (resp. affine root system), then $\mathcal{M}(S)$ is a Cartan matrix (resp. generalized Cartan matrix).

For a $\g$-module $M$, an element $\gamma\in\h$ is called an $\h$-weight of $M$, or simply a weight, if $\{x\in M\mid \ad(a)x=(a|\gamma)x,\ (a\in\h)\}\neq\{0\}$.
Note that $\h$-weights are also defined for semisimple Lie algebra by the same manner.
Let $\Pi(M)$ denote the set of $\h$-weights of $M$.
If $\dim M<\infty$, then $\Pi(M)$ is a finite set in the $\Q$-vector space spanned by roots.
This allows us to consider the matrix $\mathcal{M}(\Pi(M))$ if $0\notin \Pi(M)$ (see Definition \ref{D:M}).

Let $\sigma$ be a regular automorphism of $\g$ of order $n$.
Assume that $n$ is equal to the Coxeter number $h$ of $\g$.
Then, $\sigma$ is given by \eqref{Eq:reg}, up to conjugation, and $\h$ is a Cartan subalgebra of $\g^\sigma$.
For $i\in\Z$, set $\g_{(i)}=\{x\in\g\mid \sigma(x)=\exp((i/n)2\pi\sqrt{-1})x\}$. 
Then $\g_{(0)}=\g^\sigma$ and $\g_{(i)}$ is a finite-dimensional $\g^\sigma$-module.
\begin{lemma}\label{L:CM} Assume that $i$ is relatively prime to $n$.
Then, the matrix $\mathcal{M}(\Pi(\g_{(i)}))$ is equivalent to the generalized Cartan matrix of the affine root system of $\g$.
\end{lemma}
\begin{proof} 
By the assumption on $i$, there exists $j\in\Z$ such that $ij\equiv1\pmod n$.
Then $\sigma^j$ is also a regular automorphism of order $n$.
By Lemma \ref{Lem:Kacfpa}, $\sigma^j$ is conjugate to $\sigma$.
Hence, replacing $\sigma$ by $\sigma^j$, we may assume that $i=1$.
By \eqref{Eq:reg}, $\Pi(\g_{(1)})$ is the set of roots $\alpha$ of $\g$ such that $(\tilde{\rho}|\alpha)\in 1+h\Z$.
Hence, $\Pi(\g_{(1)})$ consists of the simple roots and the negated highest root, and we obtain the result.
\end{proof}

\subsection{Holomorphic vertex operator algebras and weight one Lie algebras}
A \emph{vertex operator algebra} (VOA) $(V,Y,\1,\omega)$ is a $\Z$-graded vector space $V=\bigoplus_{m\in\Z}V_m$ over the complex field $\C$ equipped with a linear map
$$Y(a,z)=\sum_{i\in\Z}a_{(i)}z^{-i-1}\in ({\rm End}\ V)[[z,z^{-1}]],\quad a\in V,$$
the \emph{vacuum vector} $\1\in V_0$ and the \emph{conformal vector} $\omega\in V_2$
satisfying certain axioms (\cite{Bo,FLM}). For $a\in V$ and $i\in\Z$, we call the operator $a_{(i)}$ the \emph{$i$-th mode} of $a$.
Note that the operators $L(m)=\omega_{(m+1)}$, $m\in \Z$, satisfy the Virasoro relation:
$$[L{(m)},L{(n)}]=(m-n)L{(m+n)}+\frac{1}{12}(m^3-m)\delta_{m+n,0}c\ {\rm id}_V,$$
where $c\in\C$ is called the \emph{central charge} of $V$, and $L(0)$ acts by the multiplication of a scalar $m$ on $V_m$.

A linear automorphism $g$ of a VOA $V$ is called a (VOA) \emph{automorphism} of $V$ if $$ g\omega=\omega\quad {\rm and}\quad gY(v,z)=Y(gv,z)g\quad \text{ for all } v\in V.$$
The group of all (VOA) automorphisms of $V$ will be denoted by $\Aut V$. 
A \emph{vertex operator subalgebra} (or a \emph{subVOA}) is a graded subspace of
$V$ which has a structure of a VOA such that the operations and its grading
agree with the restriction of those of $V$ and  they share the vacuum vector.
For an automorphism $g$ of a VOA $V$, let $V^g$ denote the fixed-point set of $g$.
Note that $V^g$ is a subVOA of $V$ and contains the conformal vector of $V$.

A VOA is said to be  \emph{rational} if the admissible module category is semisimple.
A rational VOA is said to be \emph{holomorphic} if there is only one irreducible module up to isomorphism.
A VOA is said to be \emph{of CFT-type} if $V_0=\C\1$ (note that $V_i=0$ for all $i<0$ if $V_0=\C\1$), and is said to be \emph{$C_2$-cofinite} if the codimension in $V$ of the subspace spanned by the vectors of form $u_{(-2)}v$, $u,v\in V$, is finite.
A module is said to be \emph{self-dual} if it is isomorphic to its contragredient module.
A VOA is said to be \emph{strongly regular} if it is rational, $C_2$-cofinite, self-dual and of CFT-type.

For $g\in\Aut V$ of order $n$, a $g$-twisted $V$-module $(M,Y_M)$ is a $\C$-graded vector space $M=\bigoplus_{m\in\C} M_{m}$ equipped with a linear map
$$Y_M(a,z)=\sum_{i\in(1/n)\Z}a_{(i)}z^{-i-1}\in (\End M)[[z^{1/n},z^{-1/n}]],\quad a\in V$$
satisfying a number of conditions (\cite{FHL,DLM2}).
We often denote it by $M$.
Note that an (untwisted) $V$-module is a $1$-twisted $V$-module
and that a $g$-twisted $V$-module is an (untwisted) $V^g$-module.
For $v\in M_k$, the \emph{conformal weight} of $v$ is $k$ and $L(0)v=kv$. 
If $M$ is irreducible, then there exists $w\in\C$ such that $M=\bigoplus_{m\in(1/n)\Z_{\geq 0}}M_{w+m}$ and $M_w\neq0$.
The number $w$ is called the \emph{conformal weight} of $M$

Let $V$ be a VOA of CFT-type.
Then, the weight one space $V_1$ has a Lie algebra structure via the $0$-th mode, which we call the \emph{weight one Lie algebra} of $V$.
 Moreover, the $n$-th modes
$v_{(n)}$, $v\in V_1$, $n\in\Z$, define  an affine representation of the Lie algebra $V_1$ on $V$.
For a simple Lie subalgebra $\mathfrak{s}$ of $V_1$, the \emph{level} of $\mathfrak{s}$ is defined to be the scalar by which the canonical central element acts on $V$ as the affine representation.
When the type of the root system of $\mathfrak{s}$ is $X_n$ and the level of $\mathfrak{s}$ is $k$, we denote the type of $\mathfrak{s}$ by $X_{n,k}$.

Assume that $V$ is self-dual.
Then there exists a non-degenerate symmetric invariant bilinear form $\langle\cdot|\cdot\rangle$ on $V$, which is unique up to scalar (\cite{Li3}).
We normalize it so that  $\langle\1|\1\rangle=-1$.
Then for any $v,w\in V_1$, we have $\langle v|w\rangle\1=v_{(1)}w$.

We, in addition, assume that the weight one Lie algebra $V_1$ is semisimple.
Let $\mathfrak{h}$ be a Cartan subalgebra of $V_1$ and let $(\cdot|\cdot)$ be the Killing form on $V_1$.
We identify $\mathfrak{h}^*$ with $\mathfrak{h}$ via $(\cdot|\cdot)$ and normalize the form so that $(\alpha|\alpha)=2$ for any long root $\alpha\in\mathfrak{h}$.
The following lemma is immediate from the commutator relations of $n$-th modes (cf.\ {\cite[(3.2)]{DM06}}).

\begin{lemma}\label{Lem:form} If the level of a simple Lie subalgebra of $V_1$ is $k$, then $\langle\cdot|\cdot\rangle=k(\cdot|\cdot)$ on it.
\end{lemma}

Let us recall some facts related to the Lie algebra $V_1$, which will be used later.

\begin{proposition}[{\cite[Theorem 1.1, Corollary 4.3]{DM06}}]\label{Prop:posl} Let $V$ be a strongly regular VOA.
Then $V_1$ is reductive.
Let $\mathfrak{s}$ be a simple Lie subalgebra of $V_1$.
Then $V$ is an integrable module for the affine representation of $\mathfrak{s}$ on $V$, and the subVOA generated by $\mathfrak{s}$ is isomorphic to the simple affine VOA associated with $\mathfrak{s}$ at some positive integral level.
\end{proposition}

\begin{proposition} [{\cite[(1.1), Theorem 3 and Proposition 4.1]{DMb}}]\label{Prop:conf} Let $V$ be a strongly regular holomorphic VOA of central charge $24$.
Then the weight one Lie algebra $V_1$ is $0$, abelian of rank $24$ or semisimple.
Moreover, the following hold:
\begin{enumerate}[{\rm (1)}]
\item If $V_1$ is abelian of rank $24$, then $V$ is isomorphic to the Leech lattice VOA.
\item If $V_1$ is semisimple, then the conformal vectors of $V$ and the subVOA generated by $V_1$ are the same.
In addition, for any simple ideal of $V_1$ at level $k$, the identity
\begin{equation}
\frac{h^\vee}{k}=\frac{\dim V_1-24}{24}\label{E:hk}
\end{equation}
holds, where $h^\vee$ is the dual Coxeter number.
\end{enumerate}
\end{proposition}

\begin{lemma}\label{L:u} Let $V$ be a strongly regular holomorphic VOA of central charge $24$.
Assume that $V_1$ is semisimple.
Let $V_1=\bigoplus_{i=1}^t \g_i$ be the decomposition of $V_1$ into the direct sum of simple ideals $\g_i$.
For $1\le i\le t$, let $h_i$ be the Coxeter number of $\g_i$ and let $\tilde\rho_i\in \g_i$ be the sum of all fundamental coweights with respect to a (fixed) set of simple roots of $\g_i$.
Set $$u=\sum_{i=1}^t\frac{1}{h_i}\tilde\rho_i,\quad \text{ and } \quad \sigma_u=\exp(-2\pi\sqrt{-1}u_{(0)})\in\Aut V.$$
Then, the restriction of $\sigma_u$ to $V_1$ is a regular automorphism whose order is the least common multiple of the Coxeter numbers $h_1,h_2,\dots,h_t$.
If all $\g_i$ are $ADE$-type, then $$\langle u|u\rangle=\frac{2\dim V_1}{\dim V_1-24}.$$
\end{lemma}
\begin{proof} Since $\tilde{\rho}_i$ and $-\tilde{\rho}_i$ are conjugate by an element in the Weyl group, the former assertion follows from Lemma \ref{Lem:Kacfpa}.

Recall the following ``strange" formula for $\g_i$ (see \cite[(13.11.4)]{Kac}):
$$(\rho_i|\rho_i)=\frac{1}{12}h_i^\vee\dim\g_i,$$
where $\rho_i$ and $h_i^\vee$ are the Weyl vector and the dual Coxeter number of $\g_i$, respectively.
Assume that all $\g_i$ are $ADE$-type.
Then $\rho_i=\tilde\rho_i$ and $h_i^\vee=h_i$ for all $i$.
Let $k_i$ be the level of $\g_i$.
By the identity \eqref{E:hk} in Proposition \ref{Prop:conf} (2) and the ``strange" formula, we obtain
$$\langle u|u\rangle=\sum_{i=1}^t\frac{1}{h_i^2}k_i(\rho_i|\rho_i)=\frac{2\dim V_1}{\dim V_1-24}$$
as desired.
\end{proof}

The following lemma is well-known, but we include a proof for completeness. 

\begin{lemma}\label{L:conjiso} Let $V$ be a VOA of CFT-type and $M$ a $V$-module.
Let $a\in V_1$ and set $g=\exp(a_{(0)})\in\Aut V$.
Then the $V$-modules $M\circ g$ and $M$ are isomorphic, where the $g$-conjugate $M\circ g=(M,Y_g)$ of $M$ is defined by $Y_g(v,z)=Y(gv,z)$ for $v\in V$.
\end{lemma}
\begin{proof} Let $f:M\to M$ be the linear isomorphism defined by $f(u)=\exp(a_{(0)})u$, $u\in M$.
Then for any $v\in V$, $u\in M$ and $n\in\Z$, we have $a_{(0)}(v_{(n)}u)=(a_{(0)}v)_{(n)}u+v_{(n)}(a_{(0)}u)$.
Hence $f(Y(v,z)u)=Y(gv,z)f(u)=Y_g(v,z)f(u)$, and $f$ is a $V$-module isomorphism from $M$ to $M\circ g$. 
\end{proof}

\subsection{Li's $\Delta$-operator}
Let $V$ be a self-dual VOA of CFT-type.
Let $u\in V_1$ such that $u_{(0)}$ acts semisimply on $V$. 
Let $\sigma_u=\exp(-2\pi\sqrt{-1}u_{(0)})$ be the inner automorphism of $V$ associated with $u$.
We assume that there exists a positive integer $T$ such that the spectrum of $u_{(0)}$ on $V$ belongs to $(1/T)\Z$. 
Then we have $\sigma_u^T=1$ on $V$.
Conversely, if $\sigma_u^T=1$, then the spectrum of $u_{(0)}$ on $V$ belongs to $(1/T)\Z$.  
Let $\Delta(u,z)$ be Li's $\Delta$-operator defined in \cite{Li}, i.e.,
\[
\Delta(u, z) = z^{u_{(0)}} \exp\left( \sum_{n=1}^\infty \frac{u_{(n)}}{-n} (-z)^{-n}\right).
\]
 
\begin{proposition}[{\cite[Proposition 5.4]{Li}}]\label{Prop:twist}
Let $\sigma$ be an automorphism of $V$ of finite order and
let $u\in V_1$ be as above such that $\sigma(u) = u$.
Let $(M, Y_M)$ be a $\sigma$-twisted $V$-module and
define $(M^{(u)}, Y_{M^{(u)}}(\cdot, z)) $ as follows:
\[
\begin{split}
& M^{(u)} =M \quad \text{ as a vector space;}\\
& Y_{M^{(u)}} (a, z) = Y_M(\Delta(u, z)a, z)\quad \text{ for } a\in V.
\end{split}
\]
Then $(M^{(u)}, Y_{M^{(u)}}(\cdot, z))$ is a
$\sigma_u\sigma$-twisted $V$-module.
Furthermore, if $M$ is irreducible, then so is $M^{(u)}$.
\end{proposition}

For a $\sigma$-twisted $V$-module $M$ and $a\in V$, we denote by $a_{(i)}^{(u)}$ the $i$-th mode of $a$ on $M^{(u)}$, i.e.,
$$Y_{M^{(u)}}(a,z)=\sum_{i\in\C}a_{(i)}^{(u)}z^{-i-1}.
$$
By the definition of Li's $\Delta$-operator, the $0$-th mode of $v\in V_1$ on $M^{(u)}$ is given by
\begin{equation}
v^{(u)}_{(0)}=v_{(0)}+\langle u|v\rangle {\rm id},\label{Eq:V1h}
\end{equation}
and the $1$-st mode of the conformal vector $\w$ on $M^{(u)}$ is given by
\begin{equation}
\omega^{(u)}_{(1)}=\omega_{(1)}+u_{(0)}+\frac{\langle u|u\rangle}{2}{\rm id}.\label{Eq:Lh}
\end{equation}

\subsection{Conformal weights of modules for simple affine VOAs}
In this subsection, we use the same notation as in Section \ref{regularauto}. 

Let $\g$ be a finite-dimensional simple Lie algebra. 
Let $Q$ be the root lattice associated with a fixed Cartan subalgebra $\h$.
For the finite-dimensional irreducible $\g$-module $M(\lambda)$ with the highest weight $\lambda$, let $\Pi(M(\lambda))$ denote the set of all $\h$-weights of $M$; we often denote it by $\Pi(\lambda)$ simply.
Since the longest element in the Weyl group maps the set of positive roots to the set of negative roots, we obtain the following lemma.

\begin{lemma}\label{Lem:min} Let $\lambda$ be a dominant integral weight.
Let $r$ be the longest element in the Weyl group.
Let $u\in\Q\otimes_\Z Q$ such that $(u|\alpha)\ge0$ for all simple roots $\alpha$.
Then
 $$\min\{(u|\mu)\mid \mu\in\Pi(\lambda)\}=(u|r(\lambda)).$$
\end{lemma}

\begin{remark}\label{Lem:longest}
\begin{enumerate}[{\rm (1)}]
\item If the type of $Q$ is $A_1$ or $D_{2n}$ $(n\ge2)$, then $r=-1$.
\item If the type of $Q$ is $A_n$ $(n\ge2)$ or $D_{2n+1}$ $(n\ge2)$, then $r$ is the product of $-1$ and the standard Dynkin diagram automorphism of order $2$.
\end{enumerate}
\end{remark}

Let $L_{\Fg}(k,0)$ be the simple affine VOA associated with $\Fg$ at a positive integral level $k$.
Let $L_{\Fg}(k,\lambda)$ be the irreducible $L_{\Fg}(k,0)$-module with the highest weight $\lambda$.
Note that $\lambda$ is a dominant integral weight of $\mathfrak{g}$ such that $(\lambda|\theta)\le k$, where $\theta$ is the highest root of $\mathfrak{g}$.
For the detail of $L_{\Fg}(k,0)$ and $L_{\Fg}(k,\lambda)$, see \cite{FZ}.
Note that the conformal weight of $L_{\Fg}(k,\lambda)$ is $\frac{(\lambda+2\rho,\lambda)}{2(k+h^\vee)}$, where $\rho$ is the Weyl vector and $h^\vee$ is the dual Coxeter number.
For $u\in \Q\otimes_\Z Q$, the inner automorphism $\sigma_u$ has finite order on $\g$ and has the same order on $L_\g(k,0)$.

\begin{lemma}\label{L:uconj} Let $M$ be a $L_{\Fg}(k,0)$-module and let $u\in \Q\otimes_\Z Q$.
\begin{enumerate}[{\rm (1)}]
\item For an element $\alpha$ of the coroot lattice $Q^\vee$, $M^{(u)}\cong M^{(u+\alpha)}$ as $\sigma_u$-twisted $L_{\Fg}(k,0)$-modules.
\item For an element $g$ of the Weyl group, the characters of $M^{(u)}$ and $M^{(g(u))}$ are the same.
\end{enumerate}
\end{lemma}
\begin{proof} Note that $\sigma_\alpha=id$ on $L_{\Fg}(k,0)$ and $(M^{(u)})^{(\alpha)}\cong M^{(u+\alpha)}$.
Then the assertion (1) was proved in \cite[Proposition 2.24]{Li01}.

Let $\hat{g}\in\Aut L_{\Fg}(k,0)$ be a lift of $g$.
Note that $\hat{g}$ is inner and $\hat{g}$ acts on $\h$ as $g$.
The $\hat{g}$-conjugate $M^{(u)}\circ \hat{g}$ (see Lemma \ref{L:conjiso} for the definition) is a $\sigma_{g(u)}$-twisted $L_{\Fg}(k,0)$-module and its character is equal to that of $M^{(u)}$.
In addition, $M^{(u)}\circ \hat{g}\cong (M\circ \hat{g})^{(g(u))}$ as $\sigma_{g(u)}$-twisted $L_{\Fg}(k,0)$-modules.
Since $\hat{g}$ is inner, we have $M\circ \hat{g}\cong M$ by Lemma \ref{L:conjiso}.
Thus we obtain (2).
\end{proof}

The following lemma gives a necessary and sufficient condition for the module $L_{\Fg}(k,\lambda)^{(u)}$ having the conformal weight zero, which corrects a mistake in \cite[Lemma 3.5]{LS16}.

\begin{lemma}[{cf. \cite[Lemma 3.5]{LS16}}]\label{Lem:lowestwt0}
Let $u\in \Q\otimes_\Z Q$ such that  
$$
(u|\beta)\ge-1\label{C:u}
$$
for any root $\beta$ of $\mathfrak{g}$.
Then the conformal weight of the irreducible $\sigma_u$-twisted $L_{\Fg}(k,0)$-module $L_{\Fg}(k,\lambda)^{(u)}$ is 
non-negative.
In addition, the conformal weight is zero if and only if 
$(\lambda,u)=(0,0)$ or $(k\eta,-g(\eta))$ for some element $g$ in the Weyl group and fundamental coweight $\eta$ with $(\eta|\theta)=1$.
\end{lemma}

\begin{remark} In fact, a fundamental coweight $\eta$ with $(\eta|\theta)=1$ is also a fundamental weight, namely, the corresponding simple root is long.
Let $g$ be an element in the Weyl group.
Then, the vector $g(\eta)-\eta$ belongs to the coroot lattice, and $(g(\eta)|\beta)\ge-1$ for any root $\beta$.
Hence, $\sigma_{-g(\eta)}=\sigma_{-\eta}=id$ on $L_{\g}(k,0)$, and $L_\g(k,k\eta)^{(-g(\eta))}\cong L_\g(k,k\eta)^{(-\eta)}\cong L_\g(k,0)$ (see \cite[Propositions 2.20]{Li01} and Lemma \ref{L:uconj} (1)).
\end{remark}

Now, let us consider the semisimple Lie algebra $\bigoplus_{i=1}^t\g_i$, where $\g_i$ are simple ideals.
For $1\le i\le t$, fix a Cartan subalgebra $\h_i$ of $\g_i$ and a set of simple roots of $\g_i$.
Let $Q_i$ be the root lattice of $\g_i$.
Let $k_i$ be a positive integer and let $\lambda_i$ be a dominant integral weight of $\g_i$ such that $(\lambda_i|\theta_i)\le k_i$, where $\theta_i$ is the highest root of $\g_i$.
Set $U=\bigotimes_{i=1}^tL_{\g_i}(k_i,0)$ and $M=\bigotimes_{i=1}^tL_{\g_i}(k_i,\lambda_i)$.

\begin{lemma}[cf. \cite{LS} and {\cite[Lemma 3.6]{LS16}}]\label{Lem:lowestwt}
Let $u_i\in \Q\otimes_\Z Q_i$ and set $u=\sum_{i=1}^tu_i$.
Assume that $$(u|\beta)\ge-1$$
for any root $\beta$ of $\bigoplus_{i=1}^t\g_i$. 
Then the conformal weight of the irreducible $\sigma_u$-twisted $U$-module $M^{(u)}$ is 
$$
w_M+\sum_{i=1}^t\min\{(u_i|\mu)\mid \mu\in\Pi(\lambda_i)\}+\frac{\langle u|u\rangle}{2},\label{Eq:twisttop}
$$
where $w_M$ is the conformal weight of $M$.
In addition, if $u$ does not belong to the coweight lattice of $\bigoplus_{i=1}^t\g_i$, then the conformal weight of $M^{(u)}$ is positive.
\end{lemma}
\begin{proof} Note that $M^{(u)}\cong \bigotimes_{i=1}^t L_{\Fg_i}(k_i,\lambda_i)^{(u_i)}$.
The former assertion follows from \cite[Lemma 3.6]{LS16}.
The latter assertion is immediate from Lemma \ref{Lem:lowestwt0}.
\end{proof}

Next we will consider several semisimple Lie algebras of ADE-type. 

\begin{lemma}\label{L:wtu} Assume that the semisimple Lie algebra $U_1=\bigoplus_{i=1}^t\g_i$ has the type 
$$A_{3,4}^3A_{1,2},\ A_{4,5}^2,\ D_{4,12}A_{2,6},\ A_{6,7},\ A_{7,4}A_{1,1}^3,\ D_{5,8}A_{1,2}\ \text{or}\ D_{6,5}A_{1,1}^2$$
and that the conformal weight of $M$ is an integer at least $2$. 
Let $u\in\h$ be the vector described in Lemma \ref{L:u}.
Let $n$ be the order of $\sigma_u$ on $U_1$.
Then the following hold:
\begin{enumerate}[{\rm (1)}]
\item $(\sigma_u)^n$ acts on $M$ as the identity operator;
\item for $j\in\Z$ with $0<|j|\le \lfloor n/2\rfloor$, the conformal weights of the $(\sigma_u)^j$-twisted $U$-modules $U^{(ju)}$ and $M^{(ju)}$ are at least $1$.
\end{enumerate}
\end{lemma}
\begin{proof} By the definition of $u$, we have $(u|\beta)\ge-1$ for any root $\beta$ of $U_1$ and $(u|\alpha)\ge0$ for any simple root $\alpha$ of $U_1$.
For all $i$, we have $Q_i=Q_i^\vee$ and $h_i=h_i^\vee$ since $\g_i$ is $ADE$-type.
Let $w_M$ be the conformal weight of $M=\bigotimes_{i=1}^tL_{\g_i}(k_i,\lambda_i)$.
By Lemmas \ref{Lem:min} and \ref{Lem:lowestwt}, the conformal weight $w_{M^{(u)}}$ of $M^{(u)}$ is
$$ w_{M^{(u)}}=w_M+\left(u\left|\sum_{i=1}^tr_i(\lambda_i)\right.\right)+\frac{\langle u|u\rangle}{2},$$
where $r_i$ is the longest element of the Weyl group of $\g_i$.

By direct computation, one can list all weights $(\lambda_1,\dots,\lambda_t)$ with $w_M\in\Z_{>1}$;
the number of such weights is $1526$, $852$, $463$, $47$, $100$, $46$ or $35$ if $V_1$ has the type $A_{3,4}^3A_{1,2}$, $A_{4,5}^2$, $D_{4,12}A_{2,6}$, $A_{6,7}$, $A_{7,4}A_{1,1}^3$, $D_{5,8}A_{1,2}$ or $D_{6,5}A_{1,1}^2$, respectively.
In addition, for every weight, one can directly check the following:
\begin{enumerate}[(i)]
\item $(\sum_{i=1}^t\lambda_i|u)\in (1/n)\Z$\label{Eq:spec};
\item $w_{M^{(u)}}\ge1$.\label{Eq:spec2}
\end{enumerate}
Note that for the cases $A_{4,5}^2$, $D_{4,12}A_{2,6}$, $A_{6,7}$ and $D_{5,8}A_{1,2}$, the claim (i) also follows from the fact that $u\in(1/n)\bigoplus_{i=1}^tQ_i$. 
See Tables \ref{T:A67}, \ref{T:D58A12} and \ref{T:D65A112} for the cases $A_{6,7}$, $D_{5,8}A_{1,2}$ and $D_{6,5}A_{1,1}^2$, respectively.
(In these tables, $[a_1,\dots,a_m]$ denotes the weight $\sum_{i=1}^m a_i\Lambda_i$.)
By Lemma \ref{L:u}, we have $w_{U^{(u)}}=\langle u|u\rangle/2\ge1$.
Then, the assertions \eqref{Eq:spec} and \eqref{Eq:spec2} prove (1) and (2) for $j=1$, respectively.

Let $j\in\Z$ with $0<|j|\le \lfloor n/2\rfloor$.
We choose a representative $\overline{ju}$ of $ju+\bigoplus_{i=1}^tQ_i$ as in Table \ref{T:rep} if  $1\le j\le \lfloor n/2\rfloor$, and $\overline{ju}=-\overline{-ju}$ if $j<0$.
Then we have $(\overline{ju}|\beta)\ge-1$ for all roots $\beta$ of $V_1$.
By Lemma \ref{L:uconj} (1), $M^{(ju)}\cong M^{(\overline{ju})}$ as $\sigma_{ju}$-twisted $U$-modules.
Let $g$ be an element in the Weyl group such that $g(\overline{ju})$ is dominant integral;
our choices of $g(\overline{ju})$, $0\le j\le \lfloor n/2\rfloor$, are summarized in Table \ref{T:rep}.
Note that $gr(-\overline{ju})=g(\overline{ju})$, where $r$ is the product of the longest elements of the Weyl groups of $\g_i$.
By Lemma \ref{L:uconj} (2), we have $w_{M^{(\overline{ju})}}=w_{M^{(g(\overline{ju}))}}$.
Hence by the same manner as in the case $j=1$, we obtain $w_{M^{(g(\overline{ju}))}}\ge1$ for $j\neq1$.
Also, by Table \ref{T:rep}, we have $w_{U^{(\overline{ju})}}=\langle g(\overline{ju})|g(\overline{ju}))\rangle/2\ge1$.
\end{proof}
\begin{tiny}
\begin{longtable}{|c|c|c|c||c|c|c|c|}
\caption{Weights of irreducible modules with $w_M\in\Z_{>1}$ for the case $A_{6,7}$}\label{T:A67}
\\
\hline
Weights& $(\sum_{i=1}^t\lambda_i|u)$& $w_M$ & $w_M^{(u)}$ 
&Weights& $(\sum_{i=1}^t\lambda_i|u)$& $w_M$ & $w_M^{(u)}$\\ \hline
$[ 0, 0, 0, 0, 0, 7]$& $3$& $3$& $2$ & 
$[ 0, 0, 0, 0, 7, 0]$&$ 5$ & $5$ & $2$ \\ \hline
$[ 0, 0, 0, 1, 0, 4]$&$ 18/7$ & $2$ & $10/7$ & 
$[ 0, 0, 0, 2, 4, 0]$&$ 32/7$ & $4$ & $10/7$ \\ \hline
$[ 0, 0, 0, 7, 0, 0]$&$ 6$ & $6$ & $2$ & 
$[ 0, 0, 1, 0, 4, 2]$&$ 32/7$ & $4$ & $10/7$ \\ \hline
$[ 0, 0, 1, 3, 0, 1]$&$ 27/7$ & $3$ & $8/7$ & 
$[ 0, 0, 2, 0, 3, 0]$&$ 27/7$ & $3$ & $8/7$ \\ \hline
$[ 0, 0, 2, 1, 0, 3]$&$ 27/7$ & $3$ & $8/7$ & 
$[ 0, 0, 2, 4, 0, 1]$&$ 39/7$ & $5$ & $10/7$ \\ \hline
$[ 0, 0, 7, 0, 0, 0]$&$ 6$ & $6$ & $2$ & 
$[ 0, 1, 0, 1, 2, 2]$&$ 27/7$ & $3$ & $8/7$ \\ \hline
$[ 0, 1, 0, 4, 2, 0]$&$ 39/7$ & $5$ & $10/7$ & 
$[ 0, 1, 2, 0, 0, 1]$&$ 20/7$ & $2$ & $8/7$ \\ \hline
$[ 0, 1, 2, 2, 1, 0]$&$ 34/7$ & $4$ & $8/7$ & 
$[ 0, 1, 3, 0, 1, 2]$&$ 34/7$ & $4$ & $8/7$ \\ \hline
$[ 0, 2, 0, 0, 2, 0]$&$ 20/7$ & $2$ & $8/7$ & 
$[ 0, 2, 0, 3, 0, 2]$&$ 34/7$ & $4$ & $8/7$ \\ \hline
$[ 0, 2, 1, 0, 3, 1]$&$ 34/7$ & $4$ & $8/7$ & 
$[ 0, 2, 4, 0, 1, 0]$&$ 39/7$ & $5$ & $10/7$ \\ \hline
$[ 0, 3, 0, 2, 0, 0]$&$ 27/7$ & $3$ & $8/7$ & 
$[ 0, 3, 1, 0, 0, 2]$&$ 27/7$ & $3$ & $8/7$ \\ \hline
$[ 0, 4, 2, 0, 0, 0]$&$ 32/7$ & $4$ & $10/7$ & 
$[ 0, 7, 0, 0, 0, 0]$&$ 5$ & $5$ & $2$ \\ \hline
$[ 1, 0, 0, 0, 2, 4]$&$ 25/7$ & $3$ & $10/7$ & 
$[ 1, 0, 0, 2, 1, 0]$&$ 20/7$ & $2$ & $8/7$ \\ \hline
$[ 1, 0, 1, 0, 1, 2]$&$ 20/7$ & $2$ & $8/7$ & 
$[ 1, 0, 1, 2, 2, 1]$&$ 34/7$ & $4$ & $8/7$ \\ \hline
$[ 1, 0, 3, 1, 0, 0]$&$ 27/7$ & $3$ & $8/7$ & 
$[ 1, 0, 4, 2, 0, 0]$&$ 39/7$ & $5$ & $10/7$ \\ \hline
$[ 1, 1, 1, 1, 1, 1]$&$ 4$ & $3$ & $1$ & 
$[ 1, 2, 0, 0, 1, 3]$&$ 27/7$ & $3$ & $8/7$ \\ \hline
$[ 1, 2, 2, 1, 0, 1]$&$ 34/7$ & $4$ & $8/7$ & 
$[ 1, 3, 0, 1, 2, 0]$&$ 34/7$ & $4$ & $8/7$ \\ \hline
$[ 2, 0, 0, 1, 3, 0]$&$ 27/7$ & $3$ & $8/7$ & 
$[ 2, 0, 0, 2, 0, 3]$&$ 27/7$ & $3$ & $8/7$ \\ \hline
$[ 2, 0, 3, 0, 2, 0]$&$ 34/7$ & $4$ & $8/7$ & 
$[ 2, 1, 0, 1, 0, 1]$&$ 20/7$ & $2$ & $8/7$ \\ \hline
$[ 2, 1, 0, 3, 1, 0]$&$ 34/7$ & $4$ & $8/7$ & 
$[ 2, 2, 1, 0, 1, 0]$&$ 27/7$ & $3$ & $8/7$ \\ \hline
$[ 2, 4, 0, 1, 0, 0]$&$ 32/7$ & $4$ & $10/7$ & 
$[ 3, 0, 1, 2, 0, 0]$&$ 27/7$ & $3$ & $8/7$ \\ \hline
$[ 3, 0, 2, 0, 0, 2]$&$ 27/7$ & $3$ & $8/7$ & 
$[ 3, 1, 0, 0, 2, 1]$&$ 27/7$ & $3$ & $8/7$ \\ \hline
$[ 4, 0, 1, 0, 0, 0]$&$ 18/7$ & $2$ & $10/7$ & 
$[ 4, 2, 0, 0, 0, 1]$&$ 25/7$ & $3$ & $10/7$ \\ \hline
$[ 7, 0, 0, 0, 0, 0]$&$ 3$ & $3$&$ 2$ &&&&\\  \hline 
\end{longtable}
\end{tiny}

\begin{tiny}
\begin{longtable}{|c|c|c|c||c|c|c|c|}
\caption{Weights of irreducible modules with $w_M\in\Z_{>1}$ for the case $D_{5,8}A_{1,2}$}\label{T:D58A12}
\\
\hline
Weights& $(\sum_{i=1}^t\lambda_i|u)$& $w_M$ & $w_M^{(u)}$ & Weights& $(\sum_{i=1}^t\lambda_i|u)$& $w_M$ & $w_M^{(u)}$\\ \hline
$([0, 0, 0, 0, 8], [0])$ & $5$ & $5$ & $2$ & 
$([0, 0, 0, 3, 3], [0])$ & $15/4$ & $3$ & $5/4$ \\ \hline
$([0, 0, 0, 4, 4], [2])$ & $11/2$ & $5$ & $3/2$ & 
$([0, 0, 0, 8, 0], [0])$ & $5$ & $5$ & $2$  \\ \hline
$([0, 0, 1, 0, 6], [2])$ & $43/8$ & $5$ & $13/8$ & 
$([0, 0, 1, 1, 3], [1])$ & $31/8$ & $3$ & $9/8$ \\ \hline
$([0, 0, 1, 3, 1], [1])$ & $31/8$ & $3$ & $9/8$ & 
$([0, 0, 1, 6, 0], [2])$ & $43/8$ & $5$ & $13/8$ \\ \hline
$([0, 0, 2, 0, 0], [2])$ & $11/4$ & $2$ & $5/4$ & 
$([0, 0, 2, 2, 2], [0])$ & $19/4$ & $4$ & $5/4$ \\ \hline
$([0, 1, 0, 1, 5], [0])$ & $37/8$ & $4$ & $11/8$ & 
$([0, 1, 0, 2, 2], [2])$ & $31/8$ & $3$ & $9/8$ \\ \hline
$([0, 1, 0, 5, 1], [0])$ & $37/8$ & $4$ & $11/8$ & 
$([0, 1, 2, 1, 1], [2])$ & $39/8$ & $4$ & $9/8$ \\ \hline
$([0, 2, 0, 0, 4], [2])$ & $19/4$ & $4$ & $5/4$ & 
$([0, 2, 0, 4, 0], [2])$ & $19/4$ & $4$ & $5/4$ \\ \hline
$([0, 3, 0, 1, 1], [0])$ & $31/8$ & $3$ & $9/8$ & 
$([1, 0, 0, 2, 4], [2])$ & $19/4$ & $4$ & $5/4$ \\ \hline
$([1, 0, 0, 4, 2], [2])$ & $19/4$ & $4$ & $5/4$ & 
$([1, 0, 1, 1, 1], [0])$ & $23/8$ & $2$ & $9/8$ \\ \hline
$([1, 0, 3, 0, 0], [0])$ & $31/8$ & $3$ & $9/8$ & 
$([1, 1, 0, 0, 2], [1])$ & $23/8$ & $2$ & $9/8$ \\ \hline
$([1, 1, 0, 1, 3], [0])$ & $31/8$ & $3$ & $9/8$ & 
$([1, 1, 0, 2, 0], [1])$ & $23/8$ & $2$ & $9/8$ \\ \hline
$([1, 1, 0, 3, 1], [0])$ & $31/8$ & $3$ & $9/8$ & 
$([1, 1, 1, 1, 1], [1])$ & $4$ & $3$ & $1$ \\ \hline
$([1, 2, 1, 0, 0], [2])$ & $31/8$ & $3$ & $9/8$ & 
$([2, 0, 0, 1, 1], [2])$ & $11/4$ & $2$ & $5/4$ \\ \hline
$([2, 0, 0, 3, 3], [0])$ & $19/4$ & $4$ & $5/4$ & 
$([2, 0, 1, 0, 2], [2])$ & $31/8$ & $3$ & $9/8$ \\ \hline
$([2, 0, 1, 1, 3], [1])$ & $39/8$ & $4$ & $9/8$ & 
$([2, 0, 1, 2, 0], [2])$ & $31/8$ & $3$ & $9/8$ \\ \hline
$([2, 0, 1, 3, 1], [1])$ & $39/8$ & $4$ & $9/8$ & 
$([2, 1, 0, 2, 2], [2])$ & $39/8$ & $4$ & $9/8$ \\ \hline
$([2, 2, 0, 0, 0], [0])$ & $11/4$ & $2$ & $5/4$ & 
$([3, 0, 0, 0, 2], [0])$ & $11/4$ & $2$ & $5/4$ \\ \hline
$([3, 0, 0, 2, 0], [0])$ & $11/4$ & $2$ & $5/4$ & 
$([3, 0, 1, 1, 1], [0])$ & $31/8$ & $3$ & $9/8$ \\ \hline
$([3, 1, 0, 0, 2], [1])$ & $31/8$ & $3$ & $9/8$  & 
$([3, 1, 0, 2, 0], [1])$ & $31/8$ & $3$ & $9/8$\\ \hline
$([4, 0, 0, 0, 0], [2])$ & $5/2$ & $2$ & $3/2$ & 
$([4, 0, 0, 1, 1], [2])$ & $15/4$ & $3$ & $5/4$ \\ \hline
$([4, 0, 2, 0, 0], [2])$ & $19/4$ & $4$ & $5/4$ & 
$([5, 0, 1, 0, 0], [0])$ & $29/8$ & $3$ & $11/8$ \\ \hline
$([6, 1, 0, 0, 0], [2])$ & $35/8$ & $4$ & $13/8$ & 
$([8, 0, 0, 0, 0], [0])$ & $4$ & $4$ & $2$  \\ \hline
\end{longtable}
\end{tiny}

\newpage
\begin{tiny}
\begin{longtable}{|c|c|c|c| |c|c|c|c|}
\caption{Weights of irreducible modules with $w_M\in\Z_{>1}$ for the case $D_{6,5}A_{1,1}^2$}\label{T:D65A112}
\\ \hline 
Weights& $(\sum_{i=1}^t\lambda_i|u)$& $w_M$ & $w_M^{(u)}$ &
Weights& $(\sum_{i=1}^t\lambda_i|u)$& $w_M$ & $w_M^{(u)}$\\ \hline
$([0, 0, 0, 0, 0, 5], [0], [1])$ & $4$ & $4$ & $3/2$ & 
$([0, 0, 0, 0, 0, 5], [1], [0])$ & $4$ & $4$ & $3/2$ \\ \hline
$([0, 0, 0, 0, 5, 0], [0], [1])$ & $4$ & $4$ & $3/2$ & 
$([0, 0, 0, 0, 5, 0], [1], [0])$ & $4$ & $4$ & $3/2$ \\ \hline
$([0, 0, 0, 1, 0, 1], [0], [1])$ & $12/5$ & $2$ & $11/10$ & 
$([0, 0, 0, 1, 0, 1], [1], [0])$ & $12/5$ & $2$ & $11/10$ \\ \hline
$([0, 0, 0, 1, 1, 0], [0], [1])$ & $12/5$ & $2$ & $11/10$ & 
$([0, 0, 0, 1, 1, 0], [1], [0])$ & $12/5$ & $2$ & $11/10$ \\ \hline
$([0, 0, 0, 1, 1, 1], [1], [1])$ & $17/5$ & $3$ & $11/10$ & 
$([0, 0, 2, 0, 0, 0], [0], [0])$ & $12/5$ & $2$ & $11/10$ \\ \hline
$([0, 0, 2, 0, 0, 1], [0], [1])$ & $17/5$ & $3$ & $11/10$ & 
$([0, 0, 2, 0, 0, 1], [1], [0])$ & $17/5$ & $3$ & $11/10$ \\ \hline
$([0, 0, 2, 0, 1, 0], [0], [1])$ & $17/5$ & $3$ & $11/10$ & 
$([0, 0, 2, 0, 1, 0], [1], [0])$ & $17/5$ & $3$ & $11/10$ \\ \hline
$([0, 1, 0, 0, 0, 2], [0], [0])$ & $12/5$ & $2$ & $11/10$ & 
$([0, 1, 0, 0, 1, 2], [0], [1])$ & $17/5$ & $3$ & $11/10$ \\ \hline
$([0, 1, 0, 0, 1, 2], [1], [0])$ & $17/5$ & $3$ & $11/10$ & 
$([0, 1, 0, 0, 2, 0], [0], [0])$ & $12/5$ & $2$ & $11/10$ \\ \hline
$([0, 1, 0, 0, 2, 1], [0], [1])$ & $17/5$ & $3$ & $11/10$ & 
$([0, 1, 0, 0, 2, 1], [1], [0])$ & $17/5$ & $3$ & $11/10$ \\ \hline
$([1, 0, 0, 1, 0, 0], [1], [1])$ & $12/5$ & $2$ & $11/10$ & 
$([1, 0, 0, 1, 1, 1], [0], [0])$ & $17/5$ & $3$ & $11/10$ \\ \hline
$([1, 0, 2, 0, 0, 0], [1], [1])$ & $17/5$ & $3$ & $11/10$ & 
$([1, 1, 0, 0, 0, 1], [0], [1])$ & $12/5$ & $2$ & $11/10$ \\ \hline
$([1, 1, 0, 0, 0, 1], [1], [0])$ & $12/5$ & $2$ & $11/10$ & 
$([1, 1, 0, 0, 0, 2], [1], [1])$ & $17/5$ & $3$ & $11/10$ \\ \hline
$([1, 1, 0, 0, 1, 0], [0], [1])$ & $12/5$ & $2$ & $11/10$ & 
$([1, 1, 0, 0, 1, 0], [1], [0])$ & $12/5$ & $2$ & $11/10$ \\ \hline
$([1, 1, 0, 0, 2, 0], [1], [1])$ & $17/5$ & $3$ & $11/10$ & 
$([2, 0, 0, 1, 0, 0], [0], [0])$ & $12/5$ & $2$ & $11/10$ \\ \hline
$([2, 0, 0, 1, 0, 1], [0], [1])$ & $17/5$ & $3$ & $11/10$ & 
$([2, 0, 0, 1, 0, 1], [1], [0])$ & $17/5$ & $3$ & $11/10$ \\ \hline
$([2, 0, 0, 1, 1, 0], [0], [1])$ & $17/5$ & $3$ & $11/10$ & 
$([2, 0, 0, 1, 1, 0], [1], [0])$ & $17/5$ & $3$ & $11/10$ \\ \hline
$([5, 0, 0, 0, 0, 0], [1], [1])$ & $3$ & $3$ & $3/2$ &&&&\\ \hline
\end{longtable}
\end{tiny}

\begin{tiny}
\begin{longtable}{|c|c|l|l|c|}
\caption{Representatives $\overline{ju}$ and dominant integral weights $g(\overline{ju})$ }\label{T:rep}\\
\hline
$U_1$& $j$& $\overline{ju}$&$g(\overline{ju})$&$\langle g(\overline{ju})|g(\overline{ju})\rangle$ \\ \hline
$A_{3,4}^3A_{1,2}$ &$1$&$u=(\frac14[1,1,1],\frac14[1,1,1],\frac14[1,1,1],\frac12[1])$&$u$&$4$\\
&$2$&$(\frac{1}{2}[-1,1,-1],\frac{1}{2}[-1,1,-1],\frac{1}{2}[-1,1,-1],[1])$&$(\frac{1}{2}[0,1,0],\frac{1}{2}[0,1,0],\frac{1}{2}[0,1,0],[1])$&$4$\\ \hline 
$A_{4,5}^2$&$1$&$u=(\frac15[1,1,1,1],\frac15[1,1,1,1])$&$u$&$4$\\
&$2$&$(\frac{1}{5}[-3,2,2,-3],\frac15[-3,2,2,-3])$&$u$&$4$\\ \hline
$D_{4,12}A_{2,6}$&$1$ &$u=(\frac16[1,1,1,1],\frac13[1,1])$&$u$&$6$\\
& $2$&$(\frac{1}{3}[1,-2,1,1],\frac{1}{3}[-1,-1])$&$(\frac{1}{3}[0,1,0,0],\frac{1}{3}[1,1])$&$4$\\ 
&$3$&$(\frac{1}{2}[1,-1,1,1],[0])$&$(\frac{1}{2}[0,1,0,0],[0])$&$6$\\ \hline
$A_{6,7}$&$1$&$u=\frac{1}{7}[1,1,1,1,1,1]$&$u$&$4$\\
&$2$&$\frac{1}{7}[2,-5,2,2,-5,2]$&$u$&$4$\\ 
&$3$&$\frac{1}{7}[3,-4,3,3,-4,3]$&$u$&$4$\\ \hline
$A_{7,4}A_{1,1}^3$&$1$ &$u=(\frac18[1,1,1,1,1,1,1],\frac12[1],\frac12[1],\frac12[1])$&$u$&$3$\\
& $2$& $\frac{1}{4}([1,-3,1,1,1,-3,1],[1],[1],[1])$&$(\frac{1}{4}[0,1,0,1,0,1,0],[1],[1],[1])$&$4$ \\
&$3$&$(\frac{1}{8}[3,3,-5,3-5,3,3],\frac{1}{2}[-1],\frac{1}{2}[-1],\frac{1}{2}[-1])$&$u$&$3$
\\ & $4$&$(\frac{1}{2}[-1,1,-1,1,-1,1,-1],[0],[0],[0])$&$(\frac{1}{2}[0,0,0,1,0,0,0],[0],[0],[0])$&$2$\\ \hline
$D_{5,8}A_{1,2}$ &$1$&$u=(\frac14[1,1,1,1,1],\frac12[1])$&$u$&$4$\\
&$2$&$(\frac{1}{4}[1,-3,1,1,1],[1])$&$(\frac{1}{4}[1,0,1,0,0],[1])$&$4$\\
&$3$& $(\frac{1}{8}[-5,3,-5,3,3],\frac{1}{2}[-1])$&$u$&$4$\\
&$4$&$(\frac{1}{2}[-1,1,-1,1,1],[0])$&$(\frac{1}{2}[0,1,0,0,0],[0])$&$4$\\ \hline
$D_{6,5}A_{1,1}^2$ &$1$&$u=(\frac1{10}[1,1,1,1,1,1],\frac12[1],\frac12[1])$&$u$&$3$\\
& $2$&$(\frac{1}{5}[1,-4,1,1,1,1],[1],[1])$&$(\frac{1}{5}[1,1,0,1,0,0],[1],[1])$&$4$\\ 
& $3$&$(\frac{1}{10}[3,3,3,-7,3,3],\frac12[-1],\frac12[-1])$&$u$&$3$\\
&$4$&$(\frac{1}{5}[2,-3,2,-3,2,2],[0],[0])$&$(\frac15[0,1,0,1,0,0],[0],[0])$&$2$\\
& $5$&$(\frac{1}{2}[1,-1,1,-1,1,1],\frac{1}{2}[1],\frac{1}{2}[1])$&$(\frac{1}{2}[0,0,1,0,0,0],\frac{1}{2}[1],\frac{1}{2}[1])$&$4$\\ \hline
\end{longtable}
\end{tiny}

\subsection{Orbifold construction and uniqueness of holomorphic VOA}

Let $V$ be a strongly regular holomorphic VOA. Suppose $\sigma\in\Aut V$ has finite order $n$.
It was proved in \cite{CM,Mi} that $V^\sigma=\{v\in V\mid \sigma(v)=v\}$ is also strongly regular.
For $1\le i\le n-1$, let $V[\sigma^i]$ be the irreducible $\sigma^i$-twisted $V$-module. By \cite{DLM2}, such a module exists and is unique up to isomorphism.
We recall the orbifold construction established in \cite{EMS}. 
\begin{theorem}[\cite{EMS}]\label{T:EMS} Assume the following:
\begin{enumerate}[{\rm (I)}]
\item for $1\le i\le n-1$, the conformal weight of $V[\sigma^i]$ is positive;
\item the conformal weight of $V[\sigma]$ belongs to $(1/n)\Z_{>0}$.
\end{enumerate}
Then, for $1\le i\le n-1$, there exists a unique irreducible $V^\sigma$-submodule $\overline{V[\sigma^i]}$ of $V[\sigma^i]$ with the integral conformal weight such that $$\widetilde{V}_\sigma:=V^\sigma\oplus\bigoplus_{i=1}^{n-1}\overline{V[\sigma^i]}$$ has a strongly regular holomorphic VOA structure as a $\Z_n$-graded simple current extension of $V^\sigma$.
\end{theorem}
This construction is often called the \emph{$\Z_n$-orbifold construction} associated with $V$ and $\sigma$.
Note that the resulting VOA $\widetilde{V}_\sigma$ is uniquely determined by $V$ and $\sigma$, up to isomorphism.
Moreover, if $\sigma'\in \Aut V$ is conjugate to $\sigma$, then $\widetilde{V}_{\sigma'}$ is isomorphic to $\widetilde{V}_{\sigma}$.

By ``reversing" the $\Z_n$-orbifold construction, the following theorem was proved in \cite{LS}.

\begin{theorem}\label{T:RO} 
Let $\g$ be a Lie algebra and $\mathfrak{p}$ a subalgebra of $\g$. 
Let $n\in \Z_{>0}$ and let $W$ be a strongly regular holomorphic VOA of central charge $c$.
Assume that for any strongly regular holomorphic VOA $V$ of central charge $c$ whose weight one Lie algebra is $\g$, there exists an order $n$ automorphism $\sigma$ of $V$ such that the following conditions hold:
\begin{enumerate}[{\rm (a)}]
\item $\g^{\sigma}\cong\mathfrak{p}$;
\item $\sigma$ satisfies Conditions (I) and (II) in Theorem \ref{T:EMS} and $\widetilde{V}_{\sigma}$ is isomorphic to $W$.
\end{enumerate}
In addition, we assume that any automorphism  $\varphi\in\Aut W$ of order $n$ satisfying (I) and (II) and the conditions (A) and (B) below belongs to a unique conjugacy class in $\Aut W$:
\begin{enumerate}[{\rm (A)}]
\item $(W^\varphi)_1$ is isomorphic to $\mathfrak{p}$;
\item $(\widetilde{W}_\varphi)_1$ is isomorphic to $\g$.
\end{enumerate}
Then any strongly regular holomorphic VOA of central charge $c$ with weight one Lie algebra $\g$ is isomorphic to $\widetilde{W}_\varphi$.
In particular, such a holomorphic VOA is unique up to isomorphism.
\end{theorem}

\begin{remark}\label{R:RO}
In general, the condition (B) is strong for the uniqueness of the conjugacy class of $\p$.
For example, the condition (B) implies that $W[\varphi]_1$ is isomorphic to $\g_{(i)}$ as $\mathfrak
{p}(\cong (W^\p)_1\cong \g^{\sigma})$-modules for some $i$ relatively prime to $n$, where $\g_{(i)}=\{x\in\g\mid \sigma(x)=\exp((i/n)2\pi\sqrt{-1})x\}$.
Later, we will consider the weaker condition:
\begin{enumerate}
\item[(B')] 
The matrices $\mathcal{M}(\Pi(W[\p]_1))$ and $\mathcal{M}(\Pi(\g_{(i)}))$ are equivalent for some $i$ relatively prime to $n$.
\end{enumerate}
In general, (B') is strictly weaker than (B) since the Lie algebra structure of $\mathfrak{g}$ may not be recovered from the $\mathfrak{g}^\sigma$-module structure of $\mathfrak{g}_{(i)}$.
\end{remark}

\subsection{Dimension formulae associated with orbifold constructions}\label{S:df}
In this subsection, we recall the dimension formulae from \cite{EMS2}.

\begin{theorem}[\cite{EMS2}] Let $n\in\{2,3,4,5,6,7,8,9,10,12,13,16,18,25\}$ and let $V$ be a strongly regular holomorphic VOA of central charge $24$.
Let $\sigma$ be an order $n$ automorphism of $V$ satisfying the conditions (I) and (II) in Theorem \ref{T:EMS}.
Assume that for $1\le i \le n-1$, the conformal weight of the irreducible $\sigma^i$-twisted $V$-module is at least $1$.
Then 
$$ \sum_{d|n}\frac{\phi((d,n/d))}{(d,n/d)}\left(24+\frac{n}{d}\dim V_1^\sigma-\dim (\tilde{V}_{\sigma^d})_1\right)=24,$$
where $\phi$ is Euler's totient function.
\end{theorem}

The explicit formulae for several $n$ are given as follows (cf. \cite{EMS2}):
\begin{itemize}
\item For $n=5,7$, $$\dim (\tilde{V}_\sigma)_1=24+(n+1)\dim V_1^\sigma-\dim V_1.$$
\item For $n=4$, $$\dim(\tilde{V}_\sigma)_1=24+6\dim V_1^\sigma -\frac{3}{2}\dim V_1^{\sigma^2}-\frac{1}{2}\dim V_1.$$
\item For $n=6$, $$\dim(\tilde{V}_\sigma)_1=24+12\dim V_1^\sigma-4\dim V_1^{\sigma^2}-3\dim V_1^{\sigma^3}+\dim V_1.$$
\item For $n=8$, $$\dim (\tilde{V}_\sigma)_1=24+12\dim V_1^\sigma-3\dim V_1^{\sigma^2}-\frac{3}{4}\dim V_1^{\sigma^4}-\frac{1}{4}\dim V_1.$$
\item For $n=10$, $$\dim (\tilde{V}_\sigma)_1=24+18\dim V_1^\sigma-6\dim V_1^{\sigma^2}-3\dim V_1^{\sigma^5}+\dim V_1.$$
\end{itemize}

\section{Leech lattice and isometries}\label{S:Leech}
Let $\Lambda$ be the Leech lattice, the unique even unimodular lattice of rank $24$ having no  norm $2$ vectors.
In this article, we adopt the notations in \cite{HL90} for the conjugacy classes of the isometry group $O(\Lambda)$ of $\Lambda$.
For $g\in O(\Lambda)$, set $\Lambda^g=\{v\in\Lambda\mid g(v)=v\}$.
It follows from Lemma \ref{L:P0} that $P_0^g(\Lambda)=(\Lambda^g)^*$ and $P_0^g(\Lambda)\subset(1/|g|)\Lambda^g$, where $P_0^g$ is the orthogonal projection (see \eqref{Eq:OP}).
Let $C_{O(\Lambda)}(g)$ denote the centralizer of $g$ in $O(\Lambda)$.
Note that the sublattices $\Lambda^g$ and the quotient groups $C_{O(\Lambda)}(g)/\langle-1\rangle$ are described in \cite[Table 1]{HL90} (cf. \cite{HM16}) and in \cite[Table 1]{Wi83}, respectively.

For $g\in O(\Lambda)$, let $p\in \Q$ such that $P_0^g(\Lambda)\subset p\Lambda^g$. For $q\in\Q$, set 
\begin{equation}
\Lambda(g,p,q):=\{x+P_0^g(\Lambda)\in p\Lambda^g/P_0^g(\Lambda) \mid (x+P_0^g(\Lambda))(q)\neq\emptyset\},\label{E:Lambdagpq}
\end{equation}
where $(x+P_0^g(\Lambda))(q)=\{y\in x+P_0^g(\Lambda)\mid (y|y)=q\}$.
In this section, we describe the $C_{O(\Lambda)}(g)$-orbits of $\Lambda(g,p,q)$ and the matrix $\mathcal{M}((x+P_0^g(\Lambda))(q))$ for some cases.
Throughout this section, let $X_n$ (resp. $\tilde{X}_n$) denote the Cartan matrix of the root system of type $X_n$ (resp. the generalized Cartan matrix of the affine root system of type $\tilde{X}_n$).
We omit the proofs of the following lemmas, which can be verified by the computer algebra system MAGMA \cite{MAGMA}.

\begin{lemma}\label{L:uni} Let $g$ be an isometry of $\Lambda$ whose conjugacy class is $4C$, $5B$, $6G$, $7B$ or $8E$.
Let $n$ be the order of $g$.
Set $$s:=\begin{cases}1/2n& (g\in 6G),\\ 1/n &(g\in 4C,5B,7B,8E).\end{cases}$$
\begin{enumerate}[{\rm (1)}]
\item The rank of $\Lambda^g$ is $10$, $8$, $6$, $6$ or $6$ if the conjugacy class of $g$ is $4C$, $5B$, $6G$, $7B$ or $8E$, respectively.
\item The minimum norm of $P_0^g(\Lambda)(=(\Lambda^g)^*)$ is $4s$.
\item If the conjugacy class of $g$ is $5B$ or $7B$, then $C_{O(\Lambda)}(g)$ acts transitively on $\Lambda(g,s,2s)$.
In addition, for $x+P_0^g(\Lambda)\in \Lambda(g,s,2s)$, the matrix $\mathcal{M}((x+P_0^g(\Lambda))(2s))$ is equivalent to the generalized Caran matrix of type $\tilde{A}_4^2$ or $\tilde{A}_6$, respectively.
\item If the conjugacy class of $g$ is $4C$, $6E$ or $8E$, then the $C_{O(\Lambda)}(g)$-orbits of $\Lambda(g,s,2s)$ are given in Tables \ref{T:4C8}, \ref{T:6G24} or \ref{T:8E16}, respectively.
\item If the conjugacy class of $g$ is $4C$ or $5B$, then the $C_{O(\Lambda)}(g)$-orbits of $\Lambda(g,s/2,2s)$ are given in Tables \ref{T:4C32} or \ref{T:5B40}, respectively.
\end{enumerate}
\end{lemma}

\begin{remark} In Tables \ref{T:6G24}, \ref{T:4C32} and \ref{T:5B40}, the matrices $\mathcal{M}((x+P_0^g(\Lambda))(2s))$ are identified only if $|(x+P_0^g(\Lambda))(2s)|=5$, $8$ and $7$, respectively, which are enough for our argument.
Indeed, we will use in Proposition \ref{P:con} the fact that the $C_{O(\Lambda)}(g)$-orbit is unique if $\mathcal{M}((x+P_0^g(\Lambda))(2s))$ is equivalent to the generalized Cartan matrix of type $\tilde{D}_4$, $\tilde{A}_7$ and $\tilde{D}_6$, respectively.
\end{remark}

\begin{tiny}
\begin{longtable}{|c|c|c|}
\caption{$C_{O(\Lambda)}(g)$-orbits of $\Lambda(g,1/4,1/2)$ for $g\in 4C$}\label{T:4C8}
\\ \hline 
 Orbit length & $|(x+P_0^g(\Lambda)(1/2)|$& $\mathcal{M}((x+P_0^g(\Lambda))(1/2))$ \\ \hline
$15$&$12$&$\tilde{A}_1^6$\\\hline
$240$&$12$&$\tilde{A}_3^3$\\ \hline 
\end{longtable}
\end{tiny}

\begin{tiny}
\begin{longtable}{|c|c|c|c|}
\caption{$C_{O(\Lambda)}(g)$-orbits of $\Lambda(g,1/12,1/6)$ for $g\in6G$}\label{T:6G24}
\\ \hline 
Orbit length & In $(1/6)\Lambda^g$? &$|(x+P_0^g(\Lambda)(1/6)|$& $\mathcal{M}((x+P_0^g(\Lambda))(1/6))$ \\ \hline
$1$&N&$6$&\\\hline
$8$&Y&$4$&\\\hline
$8$&N&$5$&$\tilde{D}_4$\\\hline
$18$&N&$2$&\\\hline
$144$&N&$3$&\\ \hline
\end{longtable}
\end{tiny}

\begin{tiny}
\begin{longtable}{|c|c|c|}
\caption{$C_{O(\Lambda)}(g)$-orbits of $\Lambda(g,1/8,1/4)$ for $g\in 8E$}\label{T:8E16}
\\ \hline 
Orbit length & $|(x+P_0^g(\Lambda)(1/4)|$& $\mathcal{M}((x+P_0^g(\Lambda))(1/4))$ \\ \hline
$3$&$6$& $\tilde{A}_1^3$\\\hline
$12$&$6$&$A_1^2\tilde{A}_3$ \\\hline
$48$&$6$&$\tilde{D}_5$\\ \hline 
\end{longtable}
\end{tiny}

\begin{tiny}
\begin{longtable}{|c|c|c|c|}
\caption{$C_{O(\Lambda)}(g)$-orbits of $\Lambda(g,1/8,1/2)$ for $g\in 4C$}\label{T:4C32}
\\ \hline 
Orbit length &In $(1/4)\Lambda^g$? & $|(x+P_0^g(\Lambda)(1/2)|$& $\mathcal{M}((x+P_0^g(\Lambda))(1/2))$ \\ \hline
$15$&Y&$12$&\\\hline
$240$&Y&$12$&\\\hline
$360$&N&$8$&$\tilde{A}_1^4$\\\hline
$1440$&N&$8$& $\tilde{A}_3^2$\\\hline
$2880$&N&$4$&\\\hline
$2880$&N&$8$&$\tilde{A}_3A_1^4$\\\hline
$11520$&N&$7$&\\\hline
$11520$&N&$5$&\\\hline
$ 15360$&N&$3$&\\\hline
$ 15360$&N&$9$&\\\hline
$23040$&N&$4$&\\\hline
$ 23040$&N&$8$&$\tilde{D}_5\tilde{A}_1$\\\hline
$ 23040$&N&$8$&$\tilde{A}_7$\\ \hline 
\end{longtable}
\end{tiny}

\begin{tiny}
\begin{longtable}{|c|c|c|c|}
\caption{$C_{O(\Lambda)}(g)$-orbits of $\Lambda(g,1/10,2/5)$ for $g\in5B$}\label{T:5B40}
\\ \hline 
 Orbit length &In $(1/5)\Lambda^g$?& $|(x+P_0^g(\Lambda)(2/5)|$& $\mathcal{M}((x+P_0^g(\Lambda))(2/5))$ \\ \hline
$75$&N&$8$&\\\hline
$144$&Y&$10$&\\\hline
$1440$&N&$2$&\\\hline
$3600$&N&$8$&\\\hline
$3600$&N&$7$&$\tilde{D}_6$\\\hline
$3600$&N&$4$&\\\hline
$7200$&N&$4$&\\ \hline 
\end{longtable}
\end{tiny}

\begin{lemma}\label{L:o2} Let $g$ be an order $2$ isometry of $\Lambda$.
\begin{enumerate}[{\rm (1)}]
\item If $g$ belongs to the conjugacy class $\pm 2A$, then $(\alpha|g(\alpha))\in2\Z$ for all $\alpha\in\Lambda$.
\item If $g$ belongs to the conjugacy class $2C$, then $(\alpha|g(\alpha))\in 1+2\Z$ for some $\alpha\in\Lambda$.
\end{enumerate}
\end{lemma}

\section{Lattice VOAs, automorphisms and twisted modules}
In this section, we review the construction of a lattice VOA and the structure of its automorphism group from \cite{FLM,DN}. We also review a construction of irreducible twisted modules for (standard) lifts of isometries from \cite{Le,DL} and study the conjugacy classes of (standard) lifts in the automorphism group of a lattice VOA.

\subsection{Lattice VOA and the automorphism group}

Let $L$ be an even lattice of rank $m$ and let $(\cdot |\cdot )$ be the positive-definite symmetric bilinear form on $\R\otimes_\Z L\cong\R^m$.
The lattice VOA $V_L$ associated with
$L$ is defined to be $M(1) \otimes \C\{L\}$. 
Here $M(1)$ is the Heisenberg VOA associated with $\mathfrak{h}=\C\otimes_\Z L$ and the form $(\cdot|\cdot)$ extended $\C$-bilinearly. That $\C\{L\}=\bigoplus_{\alpha\in L}\C e^\alpha$ is the twisted group algebra with commutator relation $e^\alpha e^\beta=(-1)^{(\alpha|\beta)}e^{\beta}e^{\alpha}$, for  $\alpha,\beta\in L$.
We fix a $2$-cocycle $\varepsilon(\cdot|\cdot):L\times L\to\{\pm1\}$ for $\C\{L\}$ such that $e^\alpha e^\beta=\varepsilon(\alpha|\beta)e^{\alpha+\beta}$, $\varepsilon(\alpha|\alpha)=(-1)^{(\alpha|\alpha)/2}$ and $\varepsilon(\alpha|0)=\varepsilon(0|\alpha)=1$ for all $\alpha,\beta\in L$.
It is well-known that the lattice VOA $V_L$ is strongly regular, and its central charge is equal to $m$, the rank of $L$.

Let $\hat{L}$ be the central extension of $L$ by $\langle-1\rangle$ associated with the $2$-cocycle $\varepsilon(\cdot|\cdot)$.
Let $\Aut \hat{L}$ be the set of all group automorphisms of $\hat L$.
For $\varphi\in \Aut \hat{L}$, we define the element $\bar{\varphi}\in\Aut L$ by $\p(e^\alpha)\in\{\pm e^{\bar\p(\alpha)}\}$, $\alpha\in L$.
Set $$O(\hat{L})=\{\p\in\Aut\hat L\mid \bar\p\in O(L)\}.$$
For $\chi\in\mathrm{Hom}(L,\Z_2)$, the map $\hat{L}\to\hat{L}$, $e^{\alpha}\mapsto (-1)^{\chi(\alpha)}e^{\alpha}$, is an element in $O(\hat{L})$.
Such automorphisms form an elementary abelian $2$-subgroup of $O(\hat{L})$ of rank $m$, which is also denoted by $\mathrm{Hom}(L,\Z_2)$ without confusion.
It was proved in \cite[Proposition 5.4.1]{FLM} that the following sequence is exact:
\begin{equation}
1 \longrightarrow \mathrm{Hom}(L, \Z_2) { \longrightarrow}
O(\hat{L}) \bar\longrightarrow O(L)\longrightarrow  1.\label{Exact1}
\end{equation}
We identify $O(\hat{L})$ as a subgroup of $\Aut V_L$ as follows: for $\varphi\in O(\hat{L})$, the map  
$$\alpha_1(-n_1)\dots\alpha_m(-n_s)e^\beta\mapsto \bar{\varphi}(\alpha_1)(-n_1)\dots\bar{\p}(\alpha_s)(-n_s)\p(e^\beta)$$
is an automorphism of $V_L$, 
where $n_1,\dots,n_s\in\Z_{>0}$ and $\alpha_1,\dots,\alpha_s,\beta\in L$.

Let $N(V_L)=\langle\exp({a_{(0)}})\mid a\in (V_L)_1\rangle,$ which is called the \emph{inner automorphism group} of $V_L$.  
We often identify $\mathfrak{h}$ with $\mathfrak{h}(-1)\1$ via $h\mapsto h(-1)\1$.
For $v\in\mathfrak{h}$, set $$\sigma_{v}=\exp(-2\pi\sqrt{-1}v_{(0)})\in N(V_L).$$ 
Note that $\sigma_v$ is the identity map of $V_L$ if and only if $v\in L^*$.
Let $$D=\{\sigma_v\mid v\in\mathfrak{h}/L^*\}\subset N(V_L).$$
Note that $\Hom(L,\Z_2)=\{\sigma_v\mid v\in (L^*/2)/L^*\}\subset D$ and that for $\p\in O(\hat{L})$ and $v\in\h$, we have $\p\sigma_v\p^{-1}=\sigma_{\bar{\p}(v)}$.

\begin{proposition}[{\cite[Theorem 2.1]{DN}}] \label{Prop:AutVLambda} The automorphism group $\Aut V_L$ of $V_L$ is generated by the normal subgroup $N(V_L)$ and the subgroup $O(\hat L)$.
\end{proposition}

At the end of this subsection, we will consider the case where $L$ has no norm $2$ vectors.
Then $(V_L)_1=\{h(-1)\1\mid h\in\mathfrak{h}\}\cong\h$, $N(V_L)=D$ and $N(V_L)\cap O(\hat L)={\rm Hom}(L,\Z_2)\cong\Z_2^{m}.$
Hence we obtain a canonical group homomorphism 
\begin{equation}
\mu:\Aut V_L\to\Aut V_L/N(V_L)\cong O(\hat L)/(O(\hat L)\cap N(V_L))\cong O(L).\label{Def:mu}
\end{equation}

\subsection{Standard lifts of isometries of lattices}
Let $L$ be an even lattice.
A (standard) lift in $O(\hat{L})$ of an isometry of $L$ is defined as follows.

\begin{definition}[\cite{Le} (see also \cite{EMS})] 
An element $\p\in O(\hat{L})$ is called a \emph{lift} of $g\in O(L)$ if $\bar{\p}=g$, where the map $\ \bar{}\ $ is defined as in \eqref{Exact1}.
A lift $\phi_g$ of $g\in O(L)$ is said to be \emph{standard} if $\phi_g(e^\alpha)=e^{\alpha}$ for all $\alpha\in L^g=\{\beta\in L\mid g(\beta)=\beta\}$.
\end{definition}

\begin{proposition}[{\cite[Section 5]{Le}}] For any isometry of $L$, there exists a standard lift.
\end{proposition}

The orders of standard lifts are determined in \cite{EMS} as follows:

\begin{lemma}[\cite{EMS}]\label{Lem:ordSLift} Let $g\in O(L)$ be of order $n$ and let $\phi_g$ be a standard lift of $g$.
\begin{enumerate}[{\rm (1)}]
\item If $n$ is odd, then the order of $\phi_g$ is also $n$.
\item Assume that $n$ is even.
Then $\phi^n_g(e^\alpha)=(-1)^{(\alpha|g^{n/2}(\alpha))}e^\alpha$ for all $\alpha\in L$.
In particular, if $(\alpha|g^{n/2}(\alpha))\in2\Z$ for all $\alpha\in L$, then the order of $\phi_g$ is $n$; otherwise the order of $\phi_g$ is $2n$.
\end{enumerate}
\end{lemma}

Next we discuss the conjugacy classes of lifts in $\Aut V_L$.
Recall that $L_g=\{\beta\in L\mid (\beta|L^g)=0\}$ and $P_0^g$ is the orthogonal projection from $\R\otimes_\Z L$ to $\R\otimes_\Z L^g$.
For the definitions of $\sigma_v$ and $D$, see the previous subsection.

\begin{lemma}\label{Lem:conjD} Let $\p\in O(\hat{L})$ and set $g=\bar\p$.
\begin{enumerate}[{\rm (1)}]
\item For $v\in\C\otimes_\Z L_{g}$, $\sigma_v\p$ is conjugate to $\p$ by an element of $D$.
\item For $v\in P_0^g(L^*)(=(L^g)^*)$, $\sigma_v\p$ is conjugate to $\p$ by an element of $D$.
\end{enumerate}
\end{lemma}
\begin{proof} Let $n$ be the order of $g$.
Since the action of $g$ on $\C\otimes_\Z L_g$ is fixed-point free, we have $\sum_{i=0}^{n-1}g^{i}=0$ on it.
Set $f=\sum_{i=1}^{n-1}ig^{i}$.
Then $(g-1)f= -\sum_{i=1}^{n-1}g^i+(n-1)id=n\cdot id$ on $\C\otimes_\Z L_g$.
For $v\in\C\otimes_\Z L_{g}$, we obtain  $$\sigma_{{f}(v/n)}(\sigma_v\p)(\sigma_{f(v/n)})^{-1}=\sigma_{v-{(g-1)}f(v/n)}\p=\p,$$
which proves (1).

Let $v\in P_0^g(L^*)$.
Then there exists $v'\in \Q\otimes L_g$ such that $v+v'\in L^*$.
Note that $\sigma_{v+v'}=id$ on $V_L$.
It follows from (1) that $\sigma_{v+v'}\p(=\p)$ is conjugate to $\sigma_{v}\p$ by an element of $D$.
Hence we obtain (2).
\end{proof}

\begin{proposition}\label{Lem:conjD2} For any isometry of $L$, its standard lift is unique up to conjugation by $D$.
\end{proposition}
\begin{proof} Let $g\in O(L)$ and let $\phi_g$ and $\phi'_g$ be standard lifts of $g$.
By the exact sequence \eqref{Exact1}, there exists $v\in L^*/2$ such that $\phi'_g=\sigma_v\phi_g$.
Since both $\phi_g$ and $\phi'_g$ are standard lifts of $g$, $\sigma_v(e^{\alpha})=e^{\alpha}$ for all $\alpha\in L^g$.
Hence $(v|L^g)\subset\Z$.
Set $v'=P_0^g(v)$.
Then $v'\in (L^g)^*$ since $(v'|\alpha)=(v|\alpha)\in\Z$ for all $\alpha\in L^g$.
By Lemma \ref{Lem:conjD} (1), $\phi'_g=\sigma_{v}\phi_g$ is conjugate to $\sigma_{v'}\phi_g$ by an element of $D$, and by Lemma \ref{Lem:conjD} (2), $\sigma_{v'}\phi_g$ is conjugate to $\phi_g$ by an element of $D$.
\end{proof}

\begin{lemma}\label{Lem:surjC} Assume that $L$ has no norm $2$ vectors.
Let $\phi_g$ be a standard lift of $g\in O(L)$ and let $\mu$ be the group homomorphism described in \eqref{Def:mu}.
\begin{enumerate}[{\rm (1)}]
\item $\mu(C_{\Aut V_L}(\phi_g))=C_{O(L)}(g)$.
\item For $\p\in\mu^{-1}(g)$, there exists $v\in \C\otimes_\Z L^g$ such that $\p$ is conjugate to $\sigma_{v}\phi_g$ by an element of $D$.
\end{enumerate}
\end{lemma}
\begin{proof} Clearly, $\mu(C_{\Aut V_L}(\phi_g))\subset C_{O(L)}(g)$.
Let $f\in C_{O(L)}(g)$ and let $\phi_f\in O(\hat L)$ be a standard lift of $f$.
Since $f$ commutes with $g$, we have $f(L^g)=L^g$.
Hence $\phi_f\phi_g\phi_f^{-1}(e^\alpha)=e^\alpha$ for $\alpha\in L^g$, that is, $\phi_f\phi_g\phi_f^{-1}$ is also a standard lift of $g$.
By Proposition \ref{Lem:conjD2}, there exists $\sigma\in D$ such that $\sigma\phi_f\phi_g\phi_f^{-1}\sigma^{-1}=\phi_g$.
Hence $\sigma\phi_f\in C_{\Aut V_L}(\phi_g)$ and $\mu (\sigma\phi_f)=f$, which proves (1).

Since the kernel of $\mu$ is $D$, there exists $v\in\mathfrak{h}$ such that $\p=\sigma_{v}\phi_g$.
Set $v'=P_0^g(v)$.
Then $v-v'\in \C\otimes_\Z L_g$.
By Lemma \ref{Lem:conjD} (1),  $\p=\sigma_{v}\phi_g$ is conjugate to $\sigma_{v'}\phi_g$ by an element of $D$.
Hence we obtain (2).
\end{proof}

\subsection{Irreducible twisted modules for lattice VOAs}\label{Sec:twist}

Let $L$ be an even unimodular lattice.
Let $g\in O(L)$ be of order $n$ and $\phi_g\in O(\hat{L})$ be a standard lift of $g$.
Then $V_L$ has a unique irreducible $\phi_g$-twisted $V_L$-module, up to isomorphism (\cite{DLM2}).
Such a module $V_L[\phi_g]$ was constructed in \cite{Le,DL} explicitly; as a vector space,
$$V_L[\phi_g]\cong M(1)[g]\otimes\C[P_0^g(L)]\otimes T,$$
where $M(1)[g]$ is the ``$g$-twisted" free bosonic space, $\C[P_0^g(L)]$ is the group algebra of $P_0^g(L)$ and $T$ is an irreducible module for a certain ``$g$-twisted" central extension of $L$. (see \cite[Propositions 6.1 and 6.2]{Le} and \cite[Remark 4.2]{DL} for detail).
Recall that $$\dim T=|L_g/(1-g)L|^{1/2}$$ and that the conformal weight $\rho_T$ of $T$ is given by  
\begin{equation}
\rho_T:=\frac{1}{4n^2}\sum_{j=1}^{n-1}j(n-j)\dim \mathfrak{h}_{(j)},\label{Eq:rho}
\end{equation}
where $\mathfrak{h}_{(j)}=\{x\in\h\mid g(x)=\exp((j/n)2\pi\sqrt{-1})x\}$.
Note that $M(1)[g]$ is spanned by vectors of the form
$$ x_1(-m_1)\dots x_s(-m_s)1,$$
where $m_i\in(1/n)\Z_{>0}$ and $x_i\in\mathfrak{h}_{(nm_i)}$ for $1\le i\le s$.

In addition, the conformal weight of $x_1(-m_1)\dots x_s(-m_s)\otimes e^\alpha\otimes t\in V_L[\phi_g]$ is given by 
$$
\sum_{i=1}^s m_i+\frac{(\alpha|\alpha)}{2}+\rho_T,\label{Eq:wtpg}
$$
where $x_1(-m_1)\dots x_s(-m_s)\in M(1)[g]$, $e^\alpha\in\C[P_0^g(L)]$ and $t\in T$.
Note that $m_i\in(1/n)\Z_{>0}$ and that the conformal weight of $V_L[\phi_g]$ is $\rho_T$.

Let $v\in\Q\otimes_\Z L^g\subset \mathfrak{h}_{(0)}$.
Then $\sigma_v$ has finite order on $V_L$ and commutes with $\phi_g$.
Note that on $(V_L)_1$, $(\alpha|\beta)=\langle\alpha|\beta\rangle$ for $\alpha,\beta\in\h$.
Let $V_L[\phi_g]^{(v)}$ be the irreducible $\sigma_v\phi_g$-twisted $V_L$-module defined as in Proposition \ref{Prop:twist}.
It is the unique irreducible $\sigma_v\phi_g$-twisted $V_L$-module, up to isomorphism, and is also denoted by $V_L[\sigma_v\phi_g]$.
By the action of $\omega_{(1)}^{(v)}$ (see \eqref{Eq:Lh}), we know that the conformal weight of $x_1(-m_1)\dots x_s(-m_s)\otimes e^\alpha\otimes t$ in $V_L[\phi_g]^{(v)}$ ($m_i\in(1/n)\Z_{>0}$, $\alpha\in P_0^g(\Lambda)$ and $t\in T$) is
\begin{equation}
\sum_{i=1}^s m_i+\frac{(\alpha|\alpha)}{2}+\rho_T +\langle v|\alpha\rangle+\frac{\langle v|v\rangle}2=\sum_{i=1}^s m_i+\frac{(v+\alpha|v+\alpha)}{2}+\rho_T.\label{Eq:wttw}
\end{equation}
Notice that the conformal weight of $V_L[\phi_g]^{(v)}$ is
$$
\frac{1}2\min\{(\beta|\beta)\mid \beta\in v+P_0^g(L)\}+\rho_T.\label{Eq:lowtw}
$$

By the explicit description of $\phi_g$-twisted and $\sigma_v\phi_g$-twisted vertex operators in \cite{Le,DL} and \eqref{Eq:V1h}, the $0$-th mode of $v\in\mathfrak{h}_{(0)}\subset (V_L^{\sigma_v\phi_g})_1$ on $V_L[\phi_g]^{(v)}$ is given by 
\begin{equation}
x^{(v)}_{(0)}(w\otimes e^\alpha\otimes t)=(x|v+\alpha)w\otimes e^\alpha\otimes t,\label{Eq:h0tw}
\end{equation}
where $w\in M(1)[g]$, $e^\alpha\in\C[P_0^g(L)]$ and $t\in T$.

Now, we assume that $\h_{(0)}$ is a Cartan subalgebra of the reductive Lie algebra $(V_L^{\sigma_v\phi_g})_1$; note that this assumption is clearly satisfied when $L$ is the Leech lattice $\Lambda$ since $(V_\Lambda^{\sigma_v\phi_g})_1=\h_{(0)}$.
Recall that  $V_L[\sigma_v\phi_g]$ is a module for $(V_L^{\sigma_v\phi_g})_1$ via the $0$-th product and $(V_L[\sigma_v\phi_g])_1$ is a submodule.  
Let $\Pi(V_L[\sigma_v\phi_g])$ (resp. $\Pi((V_L[\sigma_v\phi_g])_1)$) be the set of $\h_{(0)}$-weights of $V_L[\sigma_v\phi_g]$ (resp. $(V_L[\sigma_v\phi_g])_1$).
The equation \eqref{Eq:h0tw} above shows that the $\h_{(0)}$-weight of $w\otimes e^\alpha\otimes t$ is $v+\alpha\in\h_{(0)}$.
Then, we have
$$
\Pi(V_L[\sigma_v\phi_g])=v+P_0^g(L).\label{Eq:pi}
$$
The following lemma is immediate from \eqref{Eq:wttw}.
\begin{lemma}\label{L:hwtt} Assume that $\rho_T\ge (1-1/n)$.
\begin{enumerate}[{\rm (1)}]
\item $\Pi((V_L[\sigma_v\phi_g])_1)=(v+P_0^g(L))(2(1-\rho_T))(=\{x\in v+P_0^g(L)\mid (x|x)=2(1-\rho_T)\}).$
\item If one of the following holds, then $\Pi((V_L[\sigma_v\phi_g])_1)=\{0\}$:
\begin{enumerate}[{\rm (i)}]
\item $\rho_T\ge1$;
\item The minimum norm of $v+P_0^g(L)$ is greater than $2(1-\rho_T)$.
\end{enumerate}
\end{enumerate}
\end{lemma}

\section{Conjugacy class of the automorphism group of the Leech lattice VOA}
In this section, we study conjugacy classes of the automorphism group of the Leech lattice VOA.
We use the same notation as in Sections 3 and 4 for the Leech lattice $\Lambda$, the isometry group $O(\Lambda)$ and the Leech lattice VOA $V_\Lambda$.
Note that for (non fixed-point free) elements in $O(\Lambda)$, the characteristic polynomials and the fixed-point sublattices are summarized in \cite[Table 1]{HL90} (cf. \cite{HM16}).
Note also that the conjugacy class of a power of an element of $O(\Lambda)/\langle-1\rangle$ is described in \cite[Table 1]{Wi83}.
Recall that for $\p\in\Aut V_\Lambda$, $(V_\Lambda^\p)_1=\mathfrak{h}_{(0)}=\{v\in\mathfrak{h}\mid \mu(\p)(v)=v\}$, where $\mu:\Aut V_\Lambda\to O(\Lambda)$ is the surjective map given in \eqref{Def:mu}.
For an isometry $g$ of $\Lambda$, let $\phi_g\in\Aut V_\Lambda$ denote a standard lift of $g$ (see Section 4.2) and $\rho_T$ is the conformal weight of the subspace $T$ of $V_\Lambda[\phi_{g}]$ given in \eqref{Eq:rho}.
For the detail of the set $\Pi((V_\Lambda[\varphi])_1)$ of $\mathfrak{h}_{(0)}$-weights, see Sections \ref{regularauto} and \ref{Sec:twist}. 
For the related matrix $\mathcal{M}(\cdot)$, see Section \ref{S:lattice}.
Throughout this section, let $\tilde{X}_n$ denote the generalized Cartan matrix of the affine root system of type $\tilde{X}_n$.

\begin{longtable}{|c|c|c|c|c|c|}
\caption{$\varphi\in\Aut V_\Lambda$ and $g=\mu(\p)\in O(\Lambda)$}\label{T:CAut}
\\ \hline 
$|\varphi|$ & $\dim (V_\Lambda^\varphi)_1$& $\mathcal{M}(\Pi((V_\Lambda[\varphi])_1))$&Conjugacy class of $g$&$|\phi_g|$&$\rho_T$ \\ \hline
$4$&$10$&$\tilde{A}_3^3$&$4C$&$4$&$3/4$\\
$5$&$8$& $\tilde{A}_4^2$&$5B$&$5$&$4/5$\\
$6$&$6$&$\tilde{D}_4$&$6G$&$12$&$11/12$\\
$7$&$6$&$\tilde{A}_6$&$7B$&$7$&$6/7$\\
$8$&$10$&$\tilde{A}_7$&$4C$&$4$&$3/4$\\
$8$&$6$&$\tilde{D}_5$&$8E$&$8$&$7/8$\\
$10$&$8$&$\tilde{D}_6$&$5B$&$5$&$4/5$\\ \hline 
\end{longtable}

\begin{proposition}\label{P:con}
Let $\varphi$ be an automorphism of $V_\Lambda$.
Assume that the order $|\varphi|$ of $\varphi$, $\dim (V_\Lambda^g)_1$ and the matrix $\mathcal{M}(\Pi((V_\Lambda[\varphi])_1))$  are given as in a row of Table \ref{T:CAut}.
Then 
\begin{enumerate}[{\rm (1)}]
\item $g=\mu(\varphi)$ belongs to the conjugacy class of $O(\Lambda)$ given in the same row of Table \ref{T:CAut};
\item $\varphi$ belongs to a unique conjugacy class of $\Aut V_\Lambda$.
\end{enumerate}
\end{proposition}
\begin{proof} By Lemma \ref{Lem:surjC} (2), we may assume that $\varphi=\sigma_v\phi_g$ for some $v\in\C\otimes_\Z\Lambda^g$, up to conjugation in $\Aut V_\Lambda$.
Since $\p$ has finite order, $v\in \Q\otimes_\Z\Lambda^g$, and by Lemma \ref{Lem:conjD} (2), we may regard $u$ as an element in $(\Q\otimes_\Z\Lambda^g)/P_0^g(\Lambda)$, up to conjugation in $\Aut V_\Lambda$.
Note that $\dim (V_\Lambda^\varphi)_1$ is equal to the rank of $\Lambda^g$ and that $|g|$ divides $|\varphi|$.
Note also that $\sigma_v$ and $\phi_g$ are mutually commutative since $g(v)=v$.

By \cite[Table 1]{HL90} (cf.\ \cite[p634]{HM16}), if $(|\varphi|,\dim(V_\Lambda^\varphi)_1)=(4,10),(5,8),(7,6)$, or $(8,10)$,  then the conjugacy class of $g$ is uniquely determined as desired;
if $(|\varphi|,\dim(V_\Lambda^\varphi)_1)=(6,6)$ (resp. $(8,6)$, $(10.8)$), then the conjugacy class of $g$ is one of $\{3C,6C,-6C,-6D,6G\}$ (resp. $\{-4C,4F,8E\}$, $\{-2A,5B\}$).
We will show that it is $6G$ (resp. $8E$, $5B$) as in Table \ref{T:CAut}.
Note that by using the characteristic polynomial of $g$ (e.g. \cite[Table 1]{HL90}), one can easily compute $\rho_T$.

First, $g$ does not belong to the conjugacy class $-2A$, $3C$, $-4C$ and $-6C$; otherwise $\rho_T=1$, and by Lemma \ref{L:hwtt} (2), $\Pi((V_\Lambda[\varphi])_1)=\{0\}$, which contradicts the assumption.

Next, we suppose, for a contradiction, that $g$ belongs to the conjugacy class $4F$ (resp. $6C$, $-6D$).
Set $s=8$ (resp. $6$, $6$).
Since $g^2$ belongs to the conjugacy class $2C$ (resp. $2A$, $2A$), the order of $\phi_g$ is $s$ by Lemmas \ref{L:o2} and \ref{Lem:ordSLift}.
It follows from $\sigma_v^s=id$ that $v\in (1/s)\Lambda^g$.
Recall from \cite[Table 1]{HL90} that $\Lambda^g\cong 2\Z^{\oplus6}$ (resp. $\sqrt2E_6$, $\sqrt3E_6^*$).
Then $P_0^g(\Lambda)\cong (\Lambda^g)^*\cong (1/2)\Z^{\oplus6}$ (resp. $(1/\sqrt2)E_6^*$, $(1/\sqrt3)E_6$) and its minimum norm is $1/4$ (resp. $1/3$, $1/3$).
By $\rho_T=15/16$ (resp. $5/6$, $5/6$),  
Lemma \ref{L:hwtt} (2) and the assumption $\Pi((V_\Lambda[\varphi])_1)\neq\{0\}$,
$v+P_0^g(\Lambda)$ has vectors of norm $1/8$ (resp. $1/3$, $1/3$).
By the structure of $\Lambda^g$, the coset $v+P_0^g(\Lambda)$ has exactly $4$ (resp. $9$, $10$) vectors of norm $1/8$ (resp. $1/3$, $1/3$); notice that this number is independent of the choice of $v\in (1/s)\Lambda^g$.
This contradicts the assumption $|\Pi((V_\Lambda[\varphi])_1)|=6$ (resp. $5$, $5$) by Lemma \ref{L:hwtt} (1).
Thus we obtain (1).

Let us prove (2).
For the conjugacy classes specified in (1), $\rho_T$ and $|\phi_g|$ are summarized in Table \ref{T:CAut}.
Since $\mu$ is surjective, we may fix $g$ up to conjugation in $\Aut V_\Lambda$.
Let $n$ be the order of $g$.

We suppose, for a contradiction, that $\p=\phi_g$.
By Table \ref{T:CAut}, $\rho_T\ge (n-1)/n$.
In addition, by Lemma \ref{L:uni} (2), the minimum norm of $P_0^g(\Lambda)$ is greater than $2(1-\rho_T)$.
By Lemma \ref{L:hwtt} (2), we have $\Pi((V_\Lambda[\varphi])_1)=\{0\}$, which contradicts the assumption.
Hence $\p\neq\phi_g$, and $v\notin P_0^g(\Lambda)$.
By Lemma \ref{L:hwtt} (1), $v+P_0^g(\Lambda)$ has vectors of norm $2(1-\rho_T)$.

Let us show that $v+P_0^g(\Lambda)$ is unique, up to the action of $C_{O(\Lambda)}(g)$.
If $(|\p|,\dim (V_\Lambda^g)_1)=(6,6)$ (resp. $(8,10)$ or $(10,8)$), then $(|\phi_g|,|\p|)=(2n,n)$ (resp. $(n,2n)$) and hence $|\sigma_v|=2n$.
This implies that $v\in (1/2n)\Lambda^g$.
Since $v+P_0^g(\Lambda)$ has vectors of norm $2(1-\rho_T)$, we have $v+P_0^g(\Lambda)\in\Lambda(g,1/2n,1/n)$ (resp. $\Lambda(g,1/2n,2/n)$).
(For the definition of $\Lambda(g,p,q)$, see \eqref{E:Lambdagpq}.)
For the other cases, $|\phi_g|=|\p|=n$, and $v\in (1/n)\Lambda^g$.
Hence $v+P_0^g(\Lambda)\in \Lambda(g,1/n,2/n)$.
Then by Lemma \ref{L:uni} (3), (4), (5) (see also Tables \ref{T:4C8}, \ref{T:6G24}, \ref{T:8E16}, \ref{T:4C32} and \ref{T:5B40}), the assumption on $\mathcal{M}(\Pi((V_\Lambda[\varphi])_1))$ in Table \ref{T:CAut} and Lemma \ref{L:hwtt} (1), $v+P_0^g(\Lambda)$ is unique up to the action of $C_{O(\Lambda)}(g)$.

Since $\mu(C_{\Aut V_\Lambda}(\phi_g))=C_{O(\Lambda)}(g)$ (see Lemma \ref{Lem:surjC} (1)), 
$\p=\sigma_{v}\phi_g$ is unique, up to conjugation in $\Aut V_\Lambda$.
Therefore we obtain (2).
\end{proof}

\begin{remark}
\begin{enumerate}[(1)]
\item In the cases $(|\p|,\dim (V_\Lambda^g)_1)=(6,6), (8,10)$ and $(10,8)$, the order of $\sigma_v$ is actually $2n$; indeed $v\notin (p/n)\Lambda^g$ for any integer $p$ relatively prime $n$ by using the tables of Lemma \ref{L:uni} and the minimum norm of $v+P_0^g(\Lambda)$.
\item In the case $(|\p|,\dim (V_\Lambda^g)_1)=(6,6)$, we could check that $\phi_g^6=\sigma_v^6$ directly.
Hence the order of $\p=\sigma_v\phi_g$ is actually six.
\end{enumerate}
\end{remark}

\section{Uniqueness of holomorphic VOAs of central charge $24$}
In this section,  we will prove the main theorem using Theorem \ref{T:RO} and the Leech lattice VOA.

Let $V$ be a strongly regular holomorphic VOA of central charge $24$ whose weight one Lie algebra has the type  $A_{3,4}^3A_{1,2}$, $A_{4,5}^2$, $D_{4,12}A_{2,6}$, $A_{6,7}$, $A_{7,4}A_{1,1}^3$, $D_{5,8}A_{1,2}$ or $D_{6,5}A_{1,1}^2$.
Set $\g=V_1=\bigoplus_{i=1}^t\g_i$, where $\g_i$ are simple ideals.
Let $n$ be the least common multiple of the Coxeter numbers of $\g_i$, namely, $n=4,5,6,7,8,8$ or $10$, respectively. 
Let $k_i$ be the level of $\g_i$.
Let $\mathfrak{h}$ be a Cartan subalgebra of $\g$.
We fix a set of simple roots.
Let $\tilde\rho_i$ be the sum of all fundamental (co)weights of $\g_i$.
As in Lemma \ref{L:u}, we set 
\begin{equation}
u:=\sum_{i=1}^t\frac{1}{h_i}\tilde\rho_i,\qquad \sigma:=\sigma_u=\exp(-2\pi\sqrt{-1}u_{(0)}) \in\Aut V.\label{Def:tau7}
\end{equation}
By Lemma \ref{L:u}, we have $\langle u|u\rangle\in(2/n)\Z$.
By Lemma \ref{Lem:Kacfpa}, the restriction of $\sigma$ to $\g$ is a regular automorphism of order $n$.
Hence $V^{\sigma}_1=\g^{\sigma}=\mathfrak{h}$.

Let $U$ be the subVOA of $V$ generated by $V_1$.
By Proposition \ref{Prop:posl}, $U\cong \bigotimes_{i=1}^tL_{\g_i}(k_i,0)$.
Then $U$ is strongly regular and by Proposition \ref{Prop:conf}, the conformal vectors of $U$ and $V$ are the same.
Hence $V$ is a direct sum of finitely many irreducible $U$-submodules.

\begin{lemma}\label{L:Ord7} 
The spectrum of $u_{(0)}$ on $V$ belongs to $(1/n)\Z$ and 
the order of $\sigma$ on $V$ is $n$.
\end{lemma}
\begin{proof}
Let $M\cong \bigotimes_{i=1}^tL_{\g_i}(k_i,\lambda_i)$ be an irreducible $U$-submodule of $V$.
Since $U$ and $V$ share the same conformal vector, the conformal weight of $M$ is an integer. 
In addition, if $M\neq U$, then the conformal weight of $M$ is at least $2$ since $U_0=V_0$ and $U_1=V_1$.
By Lemma \ref{L:wtu} (1), $\sigma^n$ acts on $M$ as the identity operator.
Clearly, $\sigma$ acts on $U$ as an order $n$ automorphism.
Hence, the order of $\sigma$ is $n$ on $V$.
The assertion on the spectrum of $u_{(0)}$ has also been verified by the claim (i) in the proof of Lemma \ref{L:wtu} (1).
\end{proof}

Consider the irreducible $\sigma^j$-twisted $V$-module $V^{(ju)}$ constructed in Proposition \ref{Prop:twist} for $1\le j\le n-1$.
Note that $\sigma^j=\sigma^{j-n}$ by Lemma \ref{L:Ord7}.

\begin{proposition}\label{Prop:twist1} 
For $1\le j\le n-1$, the conformal weight of $V^{(ju)}$ belongs to $(1/n)\Z$ and is at least $1$. 
\end{proposition}
\begin{proof} By \eqref{Eq:Lh}, $\langle u|u\rangle\in(2/n)\Z$ and Lemma \ref{L:Ord7}, the conformal weight of $V^{(ju)}$ belongs to $(1/n)\Z$.
For any irreducible $U$-submodule $M$ of $V$, the conformal weight of $M^{(ju)}$ is at least $1$
(see Lemma \ref{L:wtu} (2)), and we obtain the result.
\end{proof}

By Proposition \ref{Prop:twist1}, $\sigma$ satisfies the conditions (I) and (II) of Theorem \ref{T:EMS};
let $\widetilde{V}_{\sigma}$ be the strongly regular holomorphic VOA of central charge $24$ obtained by applying the $\Z_n$-orbifold construction to $V$ and $\sigma$.

\begin{proposition}\label{Prop:abel7} The VOA $\widetilde{V}_{\sigma}$ is isomorphic to the Leech lattice VOA.
\end{proposition}
\begin{proof} For each case, one can show that $\dim (\widetilde{V}_{\sigma})_1=24$ by using the dimension formula in Section \ref{S:df} and Table \ref{T:df}; note that $U(1)$ means a $1$-dimensional abelian Lie algebra.
By Proposition \ref{Prop:conf}, $\widetilde{V}_{\sigma}$ is isomorphic to the Leech lattice VOA.
\end{proof}
\begin{table}[bht] 
\caption{Dimensions of $V^{\sigma^i}_1$}\label{T:df}
\begin{tabular}{|c|c|c|c|c|c|c|c|c|}
\hline
$n$&$V_1$ & $\dim V_1$ & $V_1^{\sigma}$& $\dim V_1^{\sigma}$& $V_1^{\sigma^2}$&$\dim V_1^{\sigma^2}$& $ V_1^{\sigma^{(n/2)}}$& $ \dim V_1^{\sigma^{(n/2)}}$\\ \hline
$4$&$A_{3,4}^3A_{1,2}$&$48$&$U(1)^{10}$& $10$&$A_1^7U(1)^3$&$24$&&\\ \hline
$5$&$A_{4,5}^2$&$48$& $U(1)^8$&$8$&&&&\\ \hline
$6$&$D_{4,12}A_{2,6}$&$36$&$U(1)^6$&$6$&$A_1^3U(1)^3$&$12$&$A_2A_1^4$&$20$\\ \hline
$7$&$A_{6,7}$& $48$& $U(1)^6$&$6$&&&&\\ \hline 
$8$&$A_{7,4}A_{1,1}^3$&$72$&$U(1)^{10}$&$10$&$A_1^7U(1)^3$&$24$&$A_3^2A_1^3U(1)$&$40$\\ \hline
$8$&$D_{5,8}A_{1,2}$&$48$&$U(1)^6$&$6$&$A_1^4U(1)^2$&$14$&$A_3A_1^3$&$24$\\ \hline
$10$&$D_{6,5}A_{1,1}^2$&$72$&$U(1)^8$&$8$&$A_1^6U(1)^2$&$20$&$A_3^2U(1)^2$&$32$\\ \hline
\end{tabular}
\end{table}

\begin{remark} For the cases $A_{3,4}^3A_{1,2}$, $A_{4,5}^2$, $D_{4,12}A_{2,6}$ and $A_{6,7}$, by using the explicit action of $V_1^{\sigma}$ on $V[\sigma^i]_1$, one could show that $(\widetilde{V}_{\sigma})_1$ is abelian, which also shows that $\widetilde{V}_{\sigma}$ is isomorphic to the Leech lattice VOA by Proposition \ref{Prop:conf}.
\end{remark}

Theorem \ref{Thm:main}, the main theorem of this article, is a corollary of the following theorem:

\begin{theorem}\label{Thm:uni7}
Let $V$ be a strongly regular holomorphic VOA of central charge $24$ such that the Lie algebra structure of $V_1$ is $A_{3,4}^3A_{1,2}$, $A_{4,5}^2$, $D_{4,12}A_{2,6}$, $A_{6,7}$, $A_{7,4}A_{1,1}^3$, $D_{5,8}A_{1,2}$ or $D_{6,5}A_{1,1}^2$.
Then $V$ is isomorphic to the holomorphic VOA $(\widetilde{V_\Lambda})_{\varphi}$, where $\varphi$ is an  automorphism in the conjugacy class of $\Aut V_\Lambda$ in Proposition \ref{P:con} (see also Table \ref{T:CAut})  with the property $(|\varphi|,\dim(V_\Lambda^\varphi)_1)=(n,\text{Lie rank of }V_1)$.
\end{theorem}
\begin{proof}
Set $\g=A_{3,4}^3A_{1,2}$, $A_{4,5}^2$, $D_{4,12}A_{2,6}$, $A_{6,7}$, $A_{7,4}A_{1,1}^3$, $D_{5,8}A_{1,2}$ or $D_{6,5}A_{1,1}^2$.
Then $n=4,5,6,7,8,8$ or $10$, respectively.
It suffices to verify the hypotheses in Theorem \ref{T:RO} for $\g$, $\mathfrak{p}=\h$ and $W=V_{\Lambda}$.
Take $u\in\h$ as in \eqref{Def:tau7} and set $\sigma=\sigma_u$.
Then the order of $\sigma$ is $n$ on $V$ by Lemma \ref{L:Ord7}.
The hypothesis (a) holds by the definition of $u$ and (b) holds by Proposition \ref{Prop:abel7}.
The hypothesis about the uniqueness of the conjugacy class follows from Proposition \ref{P:con}.
Here, we use the condition (B') in Remark \ref{R:RO}.
The detail is as follows. 
Let $i$ be an integer relatively prime to $n$ 
and let $\g_{(i)}=\{x\in\g\mid \sigma(x)=\exp((i/n)2\pi\sqrt{-1})x\}$.
Note that for a simple ideal $\mathfrak{s}$ of $\g$ whose the Coxeter number is less than $n$, we have $\mathfrak{s}\cap\g_{(i)}=\{0\}$. %
Hence $\g_{(i)}$ is contained in the ideal of $\g$ of type $A_{3}^3$, $A_{4}^2$, $D_{4}$, $A_{6}$, $A_{7}$, $D_{5}$ or $D_{6}$, respectively.
Let $\Pi(\g_{(i)})$ be the set of $\h$-weights of $\g_{(i)}$.
By Lemma \ref{L:CM}, $\mathcal{M}(\Pi(\g_{(i)}))$ is equivalent to the generalized Cartan matrix of type $\tilde{A}_3^3$, $\tilde{A}_4^2$, $\tilde{D}_4$, $\tilde{A}_6$, $\tilde{A}_7$, $\tilde{D}_5$ or $\tilde{D}_6$, respectively.
By Proposition \ref{P:con}, the conjugacy class is uniquely determined by the conditions (A) and (B').
\end{proof}

\begin{remark} In \cite{LS16b}, a strongly regular holomorphic VOA of central charge $24$ whose weight one Lie algebra has the type $A_{6,7}$ was constructed explicitly from the Leech lattice VOA and an order $7$ automorphism in the conjugacy class described in Proposition \ref{P:con}.
By the same manner, we could prove that $((\widetilde{V_\Lambda})_\varphi)_1$ is isomorphic to $\g$ directly in Theorem \ref{Thm:uni7}.
However, we omit the proof since the ``reverse" process guarantees this isomorphism.
\end{remark}

\begin{remark} In Proposition \ref{Prop:abel7}, we obtain the Leech lattice VOA by the orbifold constructions.
Considering the reverse process, we obtain holomorphic VOAs of central charge $24$ whose weight one Lie algebras have type $A_{3,4}^3A_{1,2}$, $A_{4,5}^2$, $D_{4,12}A_{2,6}$, $D_{5,8}A_{1,2}$, $A_{7,4}A_{1,1}^3$ or $D_{6,5}A_{1,1}^2$ from the Leech lattice VOA by $\Z_n$-orbifold constructions, where $n=4,5,6,8,8,10$, respectively, which are different from the previous constructions (\cite{Lam,LS16,EMS}).
Therefore we also obtain alternative constructions for these VOAs.
\end{remark}

\begin{remark}\label{R:Nils} 
Recently, the uniqueness for $15$ cases has been established in \cite{EMS2}, which includes the $3$ cases $A_{4,5}^2$, $A_{1,2}A_{3,4}^3$ and $B_{8,1}E_{8,2}$ that we have discussed in this article and in \cite{LS15}.
Up to now, there are still $9$ remaining cases for the uniqueness part:
\begin{align*}
&F_{4,6}A_{2,2},\quad E_{7,3}A_{5,1},\quad D_{7,3}A_{3,1}G_{2,1},\quad C_{4,10},\quad  A_{5,6}C_{2,3}A_{1,2},\\
&  D_{5,4}C_{3,2} A_{1,1}^2,\quad  A_{3,1} C_{7,2},\quad E_{6,4}A_{2,1}C_{2,1}\quad \text{and}\quad \emptyset.
\end{align*}
\end{remark}

\paragraph{\bf Acknowledgement.} The authors wish to thank Nils Scheithauer for sending his preprint.
They also wish to thank Sven M\"oller for pointing out a gap in the early version of this article.

\end{document}